\documentclass[11pt]{amsart}
\usepackage{amsmath,amssymb,latexsym,soul,cite,mathrsfs}

\usepackage{color,enumitem,graphicx}
\usepackage[colorlinks=true,urlcolor=blue,
citecolor=red,linkcolor=blue,linktocpage,pdfpagelabels,
bookmarksnumbered,bookmarksopen]{hyperref}
\usepackage[english]{babel}

\usepackage[left=2.9cm,right=2.9cm,top=2.8cm,bottom=2.8cm]{geometry}
\usepackage[hyperpageref]{backref}

\usepackage[colorinlistoftodos]{todonotes}
\makeatletter
\providecommand\@dotsep{5}
\def\listtodoname{List of Todos}
\def\listoftodos{\@starttoc{tdo}\listtodoname}
\makeatother

\numberwithin{equation}{section}

\newtheorem{theorem}{Theorem}[section]
\newtheorem{proposition}[theorem]{Proposition}
\newtheorem{lemma}[theorem]{Lemma}
\newtheorem{corollary}[theorem]{Corollary}

\newtheorem{remark}{Remark}

\title[On a  SBP system with Sublinear and critical nonlinearities]
{Schr\"odinger-Bopp-Podolsky system with\\  sublinear  and critical nonlinearities:\\
solutions at negative energy levels\\
and asymptotic behaviour}

\author[H. M. Santos Damian]{Heydy M. Santos Damian}
\author[G. Siciliano]{Gaetano Siciliano}

\address[H. M. S. Damian]{\newline\indent Departamento de Matem\'atica
\newline\indent 
Instituto de Matem\'atica e Estat\'istica
\newline\indent 
 Universidade de S\~ao Paulo 
\newline\indent 
Rua do Mat\~ao 1010,  05508-090, S\~ao Paulo, SP, Brazil }
\email{\href{mailto:santosh@ime.usp.br}{santosh@ime.usp.br}, 
}

\address[G. Siciliano]{\newline\indent Dipartimento di Matematica
\newline\indent 
 Universit\`a degli Studi di Bari Aldo Moro
\newline\indent 
Via E. Orabona 4, 70125 Bari, Italy}
\email{\href{mailto:gaetano.siciliano@uniba.it}{gaetano.siciliano@uniba.it}}


\begin{document}

\maketitle

\begin{abstract}
 We consider the following Schr\"odinger-Bopp-Podolsky system with critical and sublinear terms 
\begin{equation*}
\begin{cases}
   - \Delta u+ u+Q(x)\phi u= \vert u\vert^4 u+ \lambda K(x)\vert u \vert^{p-1}u&\mbox{ in }\ \mathbb{R}^3 \smallskip\\
  - \Delta \phi+ a^{2}\Delta^{2} \phi = 4\pi Q(x) u^{2}& \mbox{  in }\ \mathbb{R}^3. 
\end{cases}
\end{equation*}
Here $u,\phi:\mathbb {R}^{3}\rightarrow \mathbb{R}$ are the unknowns, $Q$  and $K$ are given functions
satisfying mild assumptions,  $a\geq0, \lambda>0$ are parameters and $p\in (0,1)$. 
  We  first show existence of  infinitely many solutions at negative energy  level, including the ground state,  when the parameter $\lambda$ is small. Then
 we give  general results concerning the structure of the set of solutions.
 
We show also the behaviour of the solutions as the parameters $a,\lambda$ tend to zero.
In particular the ground states solutions tends to a ground state solution of the Schr\"odinger-Poisson system as $a$
tends to zero.
\end{abstract}

\bigskip

\begin{center}
\begin{minipage}{12cm}
\tableofcontents
\end{minipage}
\end{center}

\bigskip

\section{Introduction}
In the last years a great attention has been given to elliptic systems which describe the interaction between matter
and electromagnetic field. Probably the first paper which has addressed this interaction by means of the Gauge Theories
has been by Benci and Fortunato, \cite{BF}. The cited paper describes the interaction of the Schr\"odinger equation with the
Maxwell equations of the electromagnetism, and has motivated and inspired many authors
in studying similar systems, involving other equations for the matter field (as e.g. the Klein-Gordon or Chern-Simons equation)
and even coupled with other equations for the e.m. field (as the Born-Infeld, or Bopp-Podolsky).

\medskip

In this paper we study a system which describes the matter field of the Schr\"odinger equation  
with  critical and sublinear nonlinearities
 interacting with  its own electromagnetic field in  the Bopp-Podolsky theory of electromagnetism.
This type of system, introduced in \cite{dAvGS}, has been studied under different assumptions, also in many other recent papers as 
\cite{CFPT,AS,dG,MS, H2, H, JLCS, LDLP, P, RS,SS}.

Specifically, the system we are interested here
is the following
 \begin{equation}\label{eq2} 
\displaystyle \left\{
   \begin{array}{lll}
   - \Delta u+ u+ Q(x)\phi u= \vert u\vert^4 u+ \lambda K(x)\vert u \vert^{p-1}u \medskip &\quad  \text{ in } \mathbb R^{3},
   \ \ p\in (0,1)\\
    - \Delta \phi+ a^{2}\Delta^{2} \phi = 4\pi Q(x)u^{2}  &\quad \text{ in } \mathbb R^{3}.
   \end{array}
 \right. 
\end{equation}
It can be obtained, as in \cite{dAvGS}, looking  at standing waves solutions under a purely electrostatic field in the generalized electrodynamics of Bopp-Podolsky. 
In \eqref{eq2} the unknowns are $u,\phi:\mathbb R^{3} \to \mathbb R$;
the nonlinearity is made by a critical term and a
sublinear one  controlled by a parameter $\lambda>0$;
 $a\geq0$ is also a parameter, known as the Bopp-Podolsky parameter; the nontrivial  functions $Q, K$ are given in suitable Lebesgue spaces, and in particular we require that

\smallskip

\begin{enumerate}[label=(Q\arabic*), ref=Q\arabic*, start=0]
\item\label{Q}  $Q\in L^{s}(\mathbb R^{3})$ for some
     $s\in[2,+\infty)$; or $s=+\infty$ and for every $t>0$ it holds $ \textrm{meas}\{x\in \mathbb R^{3}: |Q(x)|\geq t \}<+\infty$,     \end{enumerate}
    
\smallskip

\begin{enumerate}[label=(K\arabic*), ref=K\arabic*, start=0]
   \item\label{W} $K\in L^{\frac{2}{1-p}}(\mathbb{R}^{3})
   $ and 
   $\text{meas} \{ x\in \mathbb R^{3} : K(x)>0 \}>0$.
  \end{enumerate}
  
  \medskip

We will look for solutions by means of  Variational Methods. In fact the equations in the system are 
the Euler-Lagrange equations of a suitable functional defined on an Hilbert space.
So before stating our results some preliminaries are in order, especially to introduce the functional 
spaces, fix some notations and recall few known facts.

\medskip

\subsubsection*{Preliminary facts}  \ \

\medskip

In the following  $|\cdot|_{q}$ will denote the usual $L^{q}$ norm in the whole space $\mathbb R^{3}$.
On the other hand, if the set is not the whole space, let us say $\Omega\varsubsetneq \mathbb R^{3}$,
we will write $|\cdot|_{L^{q}(\Omega)}$.

Let us introduce the space
$H^{1}(\mathbb{R}^{3}) $ which is a Hilbert space equipped with the following inner product and (squared) norm:
$$\langle u,v\rangle :=\displaystyle \int_{\mathbb{R}^3 } \nabla u \nabla v+ \int_{\mathbb R^{3}}uv,
\quad \Vert u \Vert^{2}=|\nabla u|_{2}^{2} +|u|_{2}^{2}.
$$

Let $\mathcal{D}$ be the completion of $C_{c}^{\infty}(\mathbb{R}^3)$ with respect to the (squared) norm 
$\Vert.\Vert_{\mathcal{D}} ^{2}= |\nabla \varphi|_{2}^{2} +a^{2}|\Delta\varphi|_{2}^{2}$ induced by the inner product 
    $$\langle \varphi,\psi \rangle_{\mathcal{D}}:=\displaystyle \int_{\mathbb{R}^3}\nabla \varphi \nabla \psi+a^{2}\int_{\mathbb{R}^3} \Delta \varphi \Delta \psi.
$$
Then $\mathcal{D}$ is an Hilbert space continuously embedded in $D^{1,2}(\mathbb{R}^3)$ and consequently in $L^{6}(\mathbb{R}^3)$.
As proved in \cite{dAvGS} it also  embeds continuously   in $L^{\infty}(\mathbb{R}^3)$
and moreover  it is
$$\mathcal{D} = \{\phi \in D^{1,2}(\mathbb{R}^3) : \Delta \phi \in L^2(\mathbb{R}^3)\}.$$
It is easy to see that  the critical points of the  $C^{1}$ functional  
 \begin{eqnarray*}
    \Phi_{a,\lambda}(u,\phi) &=&   \frac{1}{2}\Vert u \Vert^ {2}+\frac{1}{2} \int_{\mathbb{R}^3}Q(x)\phi u^2-\frac{1}{16\pi}\vert\nabla \phi \vert^{2}_{2}-\frac{a^{2}}{16\pi}\vert \Delta \phi\vert^{2}_{2}\\
    & -&  \frac{1}{6}\int_{\mathbb{R}^3}\vert u \vert^{6} - \frac{\lambda}{p+1}\int_{\mathbb{R}^3}K(x)\vert u\vert^{p+1}
\end{eqnarray*}
in $ H^{1}(\mathbb{R}^{3}) \times \mathcal{D}$ are weak solutions of \eqref{eq2}; indeed, if $(u,\phi)\in H^{1}(\mathbb{R}^{3})\times \mathcal{D}$ is a critical point of $\Phi_{a,\lambda}$ then for every $(v,\xi)\in H^{1}(\mathbb{R}^{3})\times \mathcal D$ we have
\begin{eqnarray*}
   0&=& \partial_{u}\Phi_{a,\lambda }(u,\phi)[v]=\langle u,v\rangle +\displaystyle \int_{\mathbb{R}^3}Q(x)\phi u v - \displaystyle \int_{\mathbb{R}^3} \vert u\vert^{4}uv-\lambda \int_{\mathbb{R}^3} K(x) \vert u \vert^{p-1}u v , \\
   0&=& \partial_{\phi}\Phi_{a,\lambda }(u,\phi)[\xi]=\frac{1}{2} \int_{\mathbb{R}^3}Q(x) u^{2} \xi -\frac{1}{8\pi}\int_{\mathbb{R}^3}\nabla \phi \nabla \xi -\frac{ a^{2}}{8\pi}  \int_{\mathbb{R}^3}\Delta \phi \Delta \xi,
\end{eqnarray*}
which is exactly the weak formulation of \eqref{eq2}.

By means of a standard argument by now (introduced in \cite{BF} and used successfully
in all the subsequent literature on this subject) we can reduce ourselves to find the critical points of a functional of a single variable.
To this aim, let

$$\mathcal{K}(x) = \frac{1-e^{-\frac{\vert x \vert }{a}}}{\vert x \vert}, \ \ \ x\in \mathbb R^{3}\setminus\{0\}, a>0$$
and recall the following result given in  \cite[Lemma 3.3]{dAvGS}.

\begin{lemma}\label{lem:general}
For all $y\in \mathbb{R}^3$, $\mathcal{K}(\cdot-y)$ solves in the sense of distributions 
$$-\Delta \phi + a^{2}\Delta^2\phi = \delta_{y}.$$
Furthermore:

\begin{enumerate}
    \item[(i)] if $f \in L^1_{loc}(\mathbb{R}^3)$ and, for a.e. $x \in \mathbb{R}^3$, the map $y \in \mathbb{R}^3 \mapsto \frac{f(y)}{\vert x-y\vert}$ is summable, then $\mathcal{K}*f \in L^1_{loc}(\mathbb{R}^3)$;
    \item[(ii)] if $f \in L^p(\mathbb{R}^3)$ with $1\leq p<\frac{3}{2}$, then $\mathcal{K}*f\in L^q(\mathbb{R}^3)$ for $q \in (\frac{3p}{3-2p},+\infty ]$.
    \end{enumerate}
    In both cases, $\mathcal{K}*f$ solves
    $$-\Delta \phi +a^{2} \Delta^2\phi = f,$$
      in the sense of distributions and we have the following distributional derivatives
    $$
    \nabla(\mathcal{K}*f) = (\nabla \mathcal{K})*f \quad
    \text{ and } \quad
 \Delta(\mathcal{K} * f) = (\Delta \mathcal{K})*f \quad \text{ a.e. in } \mathbb{R}^3.
 $$
\end{lemma}

\medskip

By the Riesz Theorem there is a  unique solution in $\mathcal{D}$ of the second equation in (\ref{eq2}) which we denote by $\phi_u^{a}$.
Moreover since  $u\in H^{1}(\mathbb{R}^3)$, we have  
 $u^2 Q\in L^1(\mathbb{R}^3)$ as soon as $Q\in L^{s}(\mathbb R^{3})$ with $s\in [3/2,+\infty]$,
 so in particular under our assumption, and then $\phi_{u}^{a}\in L^{q}(\mathbb R^{3})$ for $q\in(3,+\infty]$. Finally the solution has an explicit expression
given by
\begin{equation}\label{eq:phi}
 \phi_u^{a}(x)=(\mathcal{K}*Qu^2)(x)=\int_{\mathbb{R}^3} \frac{1-e^{-\frac{\vert x-y \vert}{a}}}{4\pi \vert x-y \vert}Q(y) u^2(y)dy
 \in \mathcal D.
\end{equation}
Some properties  involving $\phi_{u}^{a}$ 
will be needed in the sequel, so we collect them here. 
The following result is proved in \cite[Lemma 3.4]{dAvGS}
and \cite[Lemma 1.2]{SS}, except for the last two convergences in \eqref{vii:lemmaphi}
that we give below. 
We denote with  $C_{0}(\mathbb R^{3})$  the space of continuous functions vanishing at infinity.

\begin{lemma}\label{lemmaphi}
For all $u\in H^1(\mathbb{R}^3)$ we have:
%
\begin{enumerate}[label=(\roman{*}), ref=\roman{*}]
    \item \label{i:lemmaphi} for every $y\in \mathbb{R}^3, \phi^{a}_{u(.+y)}=\phi^{a}_{u}(.+y); $ \smallskip
    \item \label{iii:lemmaphi}for every 
    $s \in (3,+\infty], \phi^{a}_{u}\in L^{s}(\mathbb{R}^3)\cap C_{0}(\mathbb{R}^3);$ \smallskip
    \item \label{iv:lemmaphi} for every $s \in(3/2,+\infty], \nabla \phi^{a}_{u}\in L^{s}(\mathbb{R}^3)\cap C_{0}(\mathbb{R}^3);$ \smallskip
    \item\label{v:lemmaphi}there is $C>0$ such that for all $a>0$, it is $\vert \phi^{a}_u \vert_{6}\leq C \Vert u \Vert^{2}$ and then
    $$ \|\phi_{u}^{a}\|_{\mathcal D}^{2} = 4\pi\int_{\mathbb R^{3}} Q(x)\phi^{a}_{u} u^{2} \leq C\|u\|^{4};$$
    \item\label{vi:lemmaphi}$\phi^{a}_{tu}=t^2\phi^{a}_u;$ \smallskip
    
    \item \label{viii:lemmaphi}$\phi^{a}_{u}$ is the unique minimizer of the functional 
$$\mathcal{E}_{u,a}(\phi)=\displaystyle \frac{1}{2}| \nabla \phi | ^{2}_{2}+\frac{a^{2}}{2}| \Delta \phi |^{2}_{2}-\int_{\mathbb{R}^3} Q(x) \phi u^2,\quad \phi\in \mathcal D;$$


    \item\label{vii:lemmaphi} if $v_n \rightharpoonup v $ in $H^{1}(\mathbb{R}^3),$ then \smallskip
%
%
\begin{enumerate}
\item[a)] $\phi^{a}_{v_n}\to \phi^{a}_{v}$  in $\mathcal{D}$, \smallskip
\item[b)] $\displaystyle 
\int_{\mathbb{R}^3}Q(x)\phi^{a}_{v_n} v^{2}_n   \longrightarrow \int_{\mathbb{R}^3}Q(x)\phi^{a}_{v} v^{2},  \smallskip
$
\item[c)]  $\displaystyle
\int_{\mathbb{R}^3}Q(x)\phi^{a}_{v_n} v^{2}_n \varphi  \longrightarrow \int_{\mathbb{R}^3}Q(x)\phi^{a}_{v} v^{2}\varphi , \ \ \forall \varphi \in C^{\infty}_{c}(\mathbb{R}^3), 
$ \smallskip
\item[d)]$\displaystyle
\int_{\mathbb{R}^3}Q(x)\phi^{a}_{v_n} v_n w   \longrightarrow \int_{\mathbb{R}^3}Q(x)\phi^{a}_{v} v w,  \ \ \forall w \in H^{1}(\mathbb{R}^3).
$
\end{enumerate}

\end{enumerate}
\end{lemma}

\begin{proof}
Let us prove the last two convergences of \eqref{vii:lemmaphi}.

Let $v_n \rightharpoonup v $ in $H^{1}(\mathbb{R}^3)$, so that 
$\phi_{v_{n}}^{a} \to \phi_{v}^{a}$. 
If $\varphi\in C^{\infty}_{c}(\mathbb R^{3}), F=\textrm{supp\,}\varphi$,
 then since 
 \begin{equation*}
 \frac{12s}{5s-6}\in[2,6],
 \end{equation*}
(if $s=\infty$ we agree $\frac{12s}{5s-6}=\frac{12}{5}$ and $\frac{3s}{s-3} = 3$) we get, by the H\"older inequality,
\begin{eqnarray*}
\left| \int_{\mathbb R^{3}} Q(x)\phi_{v_{n}}^{a}v_{n}^{2}\varphi - Q(x)\phi_{v}^{a}v^{2}\varphi   \right| &\leq&
\int_{F } |Q(x)|  | \phi_{v_{n}}^{a}- \phi_{v}^{a}| v_{n}^{2} |\varphi| 
+ \int_{F} |Q(x) | |v_{n}^{2}- v^{2}| |\phi_{v}^{a}|   |\varphi|\\
&\leq& |Q|_{s}   | \phi_{v_{n}}^{a}- \phi_{v}^{a}|_{6}  |v_{n} |^{2}_{\frac{12s}{5s-6}}   |\varphi|_{\infty}
+ |Q|_{s} |v_{n}^{2} - v^{2}|_{L^{2}(F)} |\phi^{a}_{v}|_{6}   |\varphi|_{\frac{3s}{s-3}}\\
&\longrightarrow& 0.
\end{eqnarray*}
This proves  \emph{vii-c)}.

Let us address the convergence in \emph{vii-d)}.
First note that in a  is similar way as before we get
\begin{equation*}\label{eq:rossa}
\int_{\mathbb{R}^3}Q(x)\phi^{a}_{v_n} v_n \varphi  \longrightarrow \int_{\mathbb{R}^3}Q(x)\phi^{a}_{v} {v}\varphi , \ \ \forall \varphi \in C^{\infty}_{c}(\mathbb{R}^3).
\end{equation*}
In fact
\begin{eqnarray*}
 \int_{\mathbb R^{3}} |Q(x) |  |\phi_{v_{n}}^{a}v_{n}  - \phi_{v}^{a} v| |\varphi | &\leq&
\int_{F } |Q(x)|  | \phi_{v_{n}}^{a}- \phi_{v}^{a}| |v_{n}| |\varphi| 
+ \int_{F} |Q(x) | |v_{n}- v| |\phi_{v}^{a}|   |\varphi|\\
&\leq& |Q|_{s}   | \phi_{v_{n}}^{a}- \phi_{v}^{a}|_{6}  |v_{n} |_{2}   |\varphi|_{\frac{3s}{s-3}}
+ |Q|_{s} |v_{n} - v|_{L^{2}(F)} |\phi^{a}_{v}|_{6}   |\varphi|_{\frac{3s}{s-3}}\\
&\longrightarrow& 0
\end{eqnarray*}
This means that, defining the linear and continuous operators on $H^{1}(\mathbb R^{3})$
$$L_{n}w:=\int_{\mathbb R^{3}} Q(x)\phi^{a}_{v_{n}}v_{n}w ,\quad  Lw:=\int_{\mathbb R^{3}} Q(x)\phi^{a}_{v}vw,$$
we have the pointwise convergence $L_{n}\varphi \to L\varphi$ on a dense set.
However the sequence of operators is bounded; in fact
since 
$$
\frac{1}{s}+\frac16 + \frac{5s-6}{12s}+ \frac{5s-6}{12s} =1 \quad \text{ and }\quad \frac{12s}{5s-6}\in[2,6],
$$
by the H\"older inequality, the boundedness of $\{\phi^{a}_{v_{n}} \}_{n\in \mathbb N}$ in $L^{6}(\mathbb R^{3})$ 
and  of $\{ v_{n}\}_{n\in \mathbb N}$ in $H^{1}(\mathbb R^{3})$ we have
$$
\| L_{n}\| = \sup_{\| w\| =1} \left| \int_{\mathbb R^{3}} Q(x) \phi^{a}_{v_{n}} v_{n} w\right|  
\leq \sup_{\| w\| =1}  | Q|_{s} |\phi^{a}_{v_{n}} |_{6} | v_{n}|_{ \frac{12s}{5s-6}} | w|_{\frac{12s}{5s-6}}   
\leq c_{1}.
$$
Hence $L_{n}w\to Lw$ which is the desired convergence.
\end{proof}

Then usual arguments (see \cite{dAvGS}) show that the energy functional of a single variable $u\in H^{1}(\mathbb{R}^3)$,
 given by
\begin{eqnarray*}\label{eq:Phi}
   \mathcal{I}_{a,\lambda}(u)&:=&\Phi_{a,\lambda}(u,\phi^{a}_{u})\nonumber\\
   &=& \frac{1}{2}\Vert u\Vert^{2}+\frac{1}
{4} \int_{\mathbb{R}^3}Q(x)\phi^{a}_u u^2-\frac{1}{6}\int_{\mathbb{R}^3}\vert u \vert^{6} - \frac{\lambda}{p+1}\int_{\mathbb{R}^3}K(x)\vert u\vert^{p+1},
  \end{eqnarray*}
is   of class $C^1$ and  for all $u,v \in H^{1}(\mathbb{R}^3) $ it is
\begin{eqnarray*}
    \mathcal{I}'_{a,\lambda}(u)[v]&=& \int_{\mathbb{R}^3} (\nabla u \nabla v + u v) +\int_{\mathbb{R}^3}  Q(x)\phi_{u}^{a}u v -\int_{\mathbb{R}^3} |u|^{4}u v -\lambda \int_{\mathbb{R}^3} K(x)\vert u \vert^{p-1}uv. 
\end{eqnarray*}
So its critical points give the solutions we are looking for, and 
solving system $(\ref{eq2})$ is equivalent to solve the equation
\begin{equation}\label{eq:eq}
    -\Delta u +  u+  Q(x)\phi^{a}_u u=  \vert u\vert^4 u+ \lambda K(x)\vert u \vert^{p-1}u\ \ \text{ in }  \mathbb{R}^3
\end{equation}
which is the one we will consider from now on.
 In fact it happens that the following statements are equivalent:
 \begin{itemize}
      \item[(a)] the pair $(u_{a,\lambda},\phi_{a,\lambda})\in H^{1}(\mathbb{R}^3) \times \mathcal{D}$ is a critical point of $\Phi_{a,\lambda}$, \smallskip
     \item[(b)] $u_{a,\lambda}\in H^{1}(\mathbb{R}^3)$ is a  critical point of $\mathcal{I}_{a,\lambda}$ and $\phi_{a,\lambda}=\phi^{a}_{u_{a,\lambda}}$ given in \eqref{eq:phi}.
 \end{itemize}

Then speaking of solutions $u_{a,\lambda}$ of \eqref{eq:eq} has to be understood as speaking of solutions
$(u_{a,\lambda},\phi^{a}_{u_{a,\lambda}})$ of \eqref{eq2}, and viceversa.

\medskip

Since we aim to study also the limit behavior as $a\to 0$ of the solutions
of \eqref{eq2}, it will be useful to consider the ``limit problem''
obtained formally by setting $a=0$.
It is exactly  the well-known Schr\"odinger-Poisson system
 \begin{equation*}\label{eq:SP}
\displaystyle \left\{
   \begin{array}{lll}
   - \Delta u+  u+Q(x)\phi u= |u|^{4}u+\lambda K(x) |u|^{p-1}u&\quad  \text{ in } \mathbb R^{3}, \medskip \\
    - \Delta \phi = 4\pi Q(x)u^{2}&\quad  \text{ in } \mathbb R^{3},
   \end{array}
 \right. 
\end{equation*}
and setting
\begin{equation*}\label{eq:phi0}
\phi_{u}^{0}(x):=(\frac{1}{|\cdot |}*Qu^2)(x)=\int_{\mathbb{R}^3} \frac{1}{ \vert x-y \vert}Q(y) u^2(y)dy
 \in D^{1,2}(\mathbb R^{3})
\end{equation*}
the system is reduced to the equation
\begin{equation}\label{eq:eqSP}
    -\Delta u + u+  Q(x)\phi^{0}_u u= |u|^{4}u+\lambda K(x)| u|^{p-1}u \ \ \text{ in }  \mathbb{R}^3.
\end{equation}
Its solutions are characterised as critical points of the energy functional on $H^{1}(\mathbb R^{3})$
\begin{equation*}\label{eq:I0}
\mathcal I_{0,\lambda} (u) = 
 \frac{1}{2}\Vert u\Vert^{2}+\frac{1}
{4} \int_{\mathbb{R}^3}Q(x)\phi^{0}_u u^2- \frac{1}{6} \int_{\mathbb{R}^3} |u|^{6}
-\frac{\lambda}{p+1}\int_{\mathbb R^{3}} K(x)|u|^{p+1}.
\end{equation*}
We just recall that also in this case there is some $C>0$ such that  $ \|\phi_{u}^{0}\|_{D^{1,2}} \leq C\|u\|^{2}$
and $u_{n}\rightharpoonup u$ in $H^{1}(\mathbb R^{3})$ implies $\phi_{u_{n}}^{0} \to \phi^{0}_{u}$ in $D^{1,2}(\mathbb R^{3})$,
see e.g. \cite{CM} where this converge is proved for $Q\in L^{2}(\mathbb R^{3})$ but actually holds under assumption \eqref{Q}.

Finally, observe two  useful facts true for $a\geq0$
hence for the Schr\"odinger-Poisson and Schr\"odinger-Bopp-Podolsky systems:
(i) it is
$$\int_{\mathbb R^{3}} Q(x)\phi^{a}_{u} u^{2} \geq0$$ 
and (ii) the functionals  $\mathcal I_{a,\lambda}$  are unbounded from below,
not coercive and even. In particular their critical points will appear in pairs.

Between the solutions we will find, there are the ones of minimal energy. These solutions are named
{\sl ground states}. 

\section{Main results}
We are able now to state our theorems which are of two types.
The first type concerns the existence of
 solutions that we find with the variational method,
by using the genus of Krasnoselsky. 
The second type are indeed results  true for any possible variational solution.

We remark that all our results are new even for the Schr\"odinger-Poisson system.

\medskip

\subsubsection{Existence of solutions} \ \

\medskip

About the existence of infinitely many solutions we have  the following results.
\begin{theorem} \label{vegeta}
Assume \eqref{Q} and \eqref{W}.
There exists a positive constant $$ \Lambda=\Lambda(p,|K|_{\frac{2}{1-p}})>0$$
  such that,
for any $\lambda \in (0,\Lambda)$ and for any $a>0$,
equation \eqref{eq:eq} possesses infinitely many solutions $\{u_{a,\lambda,n}\}_{n\in \mathbb N}\subset H^{1}(\mathbb R^{3})$,
which satisfy:
\begin{enumerate}
\item[1.]\emph{ground state:}  $u_{a,\lambda, 1}\geq0$ is the ground state;
\smallskip
\item[2.]\emph{increasing negative levels:} for any $n\in \mathbb N$ it is $ \mathcal I_{a,\lambda}(u_{a,\lambda,n})  \leq \mathcal I_{a,\lambda}(u_{a,\lambda,n+1}) <0 $ 
and 
$$\lim_{n\to+\infty} \mathcal I_{a,\lambda}(u_{a,\lambda,n}) = 0, \  \ \lim_{n\to +\infty} u_{a,\lambda,n} = 0 \ \text{ in }H^{1}(\mathbb R^{3});$$

\item[3.] \emph{uniform estimates:} the following estimate holds
$$\|u_{a,\lambda, n} \|^{1-p}  \leq \lambda \frac{5-p}{2(p+1)} |K|_{\frac{2}{1-p}};$$
\item[4.] \emph{squeezing in $\lambda$:} we have
$$\lim_{\lambda\to 0} \sup_{a>0, n\in \mathbb N}\|u_{a, \lambda, n}\| = 0;
$$
\item[5.] \emph{compactness:} the sequence $\{u_{a,\lambda,n}\}_{n\in \mathbb N}$  is compact,  in the sense that the limit in $H^{1}(\mathbb R^{3})$ of any sequence of its elements  exists up to subsequences  and is a solution too
(with the same values of $a$ and $\lambda$).
\end{enumerate}
\end{theorem}

Of course, if  $\phi_{a, \lambda, n} := \phi^{a}_{u_{a,\lambda, n}}$ are given as in \eqref{eq:phi},
in virtue of \eqref{v:lemmaphi} of Lemma \ref{lemmaphi}, it follows also
$$\lim_{n\to +\infty} \phi_{a,\lambda,n} = 0 \  \text{ in } \mathcal D
\ \  \text{ and } \ \ \lim_{\lambda\to 0} \sup_{a>0, n\in \mathbb N}\|\phi _{a, \lambda, n}\| = 0.
$$
Note also that  the squeezing in $\lambda$ is in accordance with the fact that the equation has
just the trivial solution when $\lambda=0$, 
 as one can see by means of a Pohozaev Identity.

\medskip

However by analyzing the proofs of the above theorem, we see that indeed the results hold for the
limit problem. Hence we have explicitly

\begin{theorem}\label{th:a=0}
Assume \eqref{Q} and \eqref{W}.
A similar result to Theorem \ref{vegeta}, with the same value of $\Lambda$,
holds when $a=0$, namely for the Schr\"odinger-Poisson system.

In accordance, these particular solutions are denoted with $u_{0,\lambda,n}$.
\end{theorem}

To prove the above theorems we will use variational methods and apply results in  Critical Point Theory 
to the functional $\mathcal I_{a,\lambda}$ (or $\mathcal I_{0,\lambda}$).
In particular a suitable minimax scheme, joint with a compactness condition, will guarantee 
the existence of critical points.
For this reason the solutions obtained in Theorem \ref{vegeta} and Theorem \ref{th:a=0} 
will be called of {\sl minimax type}. 

\medskip

Now, the problems may have other solutions,
both  in case $\lambda\in(0,\Lambda)$ and for other values of the parameter $\lambda$.


In the following, for the {\sl generic} solution of \eqref{eq:eq} (resp. \eqref{eq:eqSP}) 
we will use also the notation  $u_{a,\lambda}$ (resp. $u_{0,\lambda}$), in case we need to specify the parameters. 
Hence  $u_{a,\lambda,n}$ and  $u_{0,\lambda,n}$
found in Theorem \ref{vegeta} and Theorem \ref{th:a=0}, are just particular cases of these more generic solutions,
and indeed the subscript $n$
will be deserved just for the infinitely many solutions found above.

But then the following questions are quite natural.
\begin{itemize}

\item[a)] What is the behavior of the solutions with respect to $\lambda$?

\item[b)] Fixed $\lambda$ and $n$,  under  which conditions the sequence
 $ \{u_{a,\lambda,n}\}_{a\in (0,1]} $, solutions  of \eqref{eq:eq}, converges when $a$ tends to $0$? If so,
is the limit a solutions of \eqref{eq:eqSP}? Is  it exactly $u_{0,\lambda,n}$ given in Theorem \ref{th:a=0}
(namely do the solutions ``correspond'' in some sense?), or the limit  is just a generic solution $u_{0,\lambda}$?
\end{itemize}

In other words, in the point b) we are interested if the solutions of the Schr\"odinger-Bopp-Podolsky system found as minimax
by means of the Krasnoselsky genus (solutions given in Theorem \ref{vegeta}), converge, as $a$ goes to $0$, to the solution of the Schr\"odinger-Poisson equation obtained again
by means of the Krasnoselsky genus theory, and not to other kind of solution.

\medskip

We give a partial answer to the above questions. To do that we study the set of all possible solutions.

%

\medskip

\subsubsection{Structure of the solution set} \ \

\medskip

A careful look at the proofs and estimates used to prove Theorem \ref{vegeta} and Theorem \ref{th:a=0} 
also permits to obtain results about all the possible solutions. 
Define the sets of all the solutions:
\begin{eqnarray*}
&&\mathcal S_{a,\lambda}:=\{\textrm{solutions in $H^{1}(\mathbb R^{3})$ of} \  \eqref{eq:eq}\}\quad a>0,\\
&&\mathcal S_{0,\lambda}:=\{\textrm{solutions in $H^{1}(\mathbb R^{3})$ of} \  \eqref{eq:eqSP}\}.
\end{eqnarray*}
The next, is  a first general result on a  bound from below on the value of the energy on the solutions
(for both  system \eqref{eq:eq} and \eqref{eq:eqSP}).
\begin{theorem}\label{th:}
Assume \eqref{Q} and \eqref{W}. There is  a constant $\mathfrak c_{p} >0$
such that
$$\forall a\geq0,\lambda>0 \ \ \text{with}\ \ \mathcal S_{a,\lambda} \neq\emptyset : \ 
\inf_{\mathcal S_{a,\lambda}} \mathcal I_{a,\lambda}   \geq  - \lambda^{\frac{2}{1-p}}  |K|_{\frac{2}{1-p}}\mathfrak c_{p}.$$
Moreover $\mathcal I_{a,\lambda}$ is coercive on  $\mathcal S_{a,\lambda}$ if this set is unbounded
in $H^{1}(\mathbb R^{3})$.
\end{theorem}
Actually we know that for $\lambda\in(0,\Lambda)$ it is $0>\min_{\mathcal S_{a,\lambda}}\mathcal I_{a,\lambda}= \mathcal I_{a,\lambda} (u_{a,\lambda,1})$.
The next  result  states that the properties of the minimax solutions 
found in Theorem \ref{vegeta} and Theorem \ref{th:a=0} are indeed true for more general solutions.
\begin{theorem}\label{th:a>0}
Assume \eqref{Q} and \eqref{W}.
Let $a\geq0,\lambda>0$ be fixed such that $\mathcal S_{a,\lambda} \neq\emptyset$.
Define, for $c\geq\inf_{\mathcal S_{a,\lambda}}\mathcal I_{a,\lambda}$, 
 the sublevel set of solutions
$$\mathcal S_{a,\lambda}^{c}:=\{u_{a,\lambda}\in \mathcal S_{a,\lambda} :  \ \mathcal I_{a,\lambda}(u_{a,\lambda})\leq c\}.$$

Then
\begin{itemize}
\item[a)] $\mathcal S_{a,\lambda}^{c}$ is bounded; more specifically
\begin{eqnarray*}\label{eq:}
&&\forall u_{a,\lambda} \in \mathcal S_{a,\lambda}^{c} : \ \ 
\|u_{a,\lambda}\|^{2} \leq \lambda\frac{5-p}{2(p+1)} |K|_{\frac{2}{1-p}}  \|u_{a,\lambda} \|^{p+1} + 3c; 
\end{eqnarray*}
\item[b)] there is a  constant $\mathfrak a_{p}>0$ such that if $c<\frac13 S^{\frac32}- \lambda^{\frac{1}{2-p}}|K|_{\frac{2}{1-p}}^{\frac{2}{1-p}}
 \mathfrak a_{p}$  then $\mathcal S_{a,\lambda}^{c} $ is compact and, besides the estimate in $a$) we have also the following bound:
 $$\frac{1}{4}|\nabla u_{a,\lambda}|_{2}^{2} + \frac{1}{12}|u_{a,\lambda}|_{6}^{6} < \frac{1}{3}S^{3/2}.$$
\end{itemize}
\end{theorem}
As we will see, the constants  $\mathfrak c_{p}$  and 
$\mathfrak a_{p}$  of Theorem \ref{th:} and Theorem \ref{th:a>0}
are explicit and just depend on $p$.
\medskip

\subsubsection{Behavior of the solutions as the parameters tend to $0$} \ \ 

\medskip

It follows by Theorem \ref{th:a>0} that, for $a$ fixed, the solutions are bounded as  $\lambda$ tends to zero.
However for solutions at nonpositive energy level,
as the one in Theorem \ref{vegeta} and Theorem \ref{th:a=0}, 
 more can be said.
\begin{theorem}\label{th:squeezing}
Assume \eqref{Q} and \eqref{W}.
When $c\leq0$, for the solutions in $\mathcal S_{a,\lambda}^{c}$
the squeezing in $\lambda$ (uniformly in $a$) follows:
\begin{eqnarray*}\label{eq:squeeze}
&&\lim_{\lambda\to 0} \sup_{a>0} \|u_{a, \lambda}\|= 0\quad \text{and}\quad 
\lim_{\lambda\to 0} \|u_{0, \lambda}\| = 0.
\end{eqnarray*}
Then, in particular it follows that
\begin{eqnarray}
&&\lim_{\lambda\to 0}\mathcal I_{a,\lambda}(u_{a,\lambda}) = 
\lim_{\lambda\to 0}\mathcal I_{0,\lambda}(u_{0,\lambda}) =
\lim_{\lambda\to 0}\mathcal I_{a,\lambda}(u_{0,\lambda}) 
= \lim_{\lambda\to 0}\mathcal I_{0,\lambda}(u_{a,\lambda}) =0, \label{eq:new0}\\
&& \lim_{\lambda\to 0}\mathcal I'_{a,\lambda}(u_{0,\lambda}) =
\lim_{\lambda\to 0}\mathcal I'_{0,\lambda}(u_{a,\lambda})  =0\ \quad \text{ in } H^{-1}(\mathbb R^{3}). \label{eq:new}
\end{eqnarray}
\end{theorem}

\medskip

%

Finally we give also an asymptotic behavior of the solutions  of \eqref{eq:eq} whenever the parameter $\lambda$
is fixed and the parameter 
$a$ tends to zero, obtaining once again (as in previous papers \cite{dAvGS, FS,LS,SS}) that the Schr\"odinger-Poisson system 
can be seen in some sense as limit of the Schr\"odinger-Bopp-Podolsky system. In the next results
 $\lambda$ has to be considered fixed in such a way that both \eqref{eq:eq} and \eqref{eq:eqSP} admit solutions.
Again we state the result for general solutions, not only the particular  sequence $\{u_{a,\lambda,n}\}_{n\in \mathbb N}$ found in Theorem 
\ref{vegeta} and Theorem \ref{th:a=0}.
\begin{theorem}\label{th:convina}
Assume \eqref{Q} and \eqref{W}.
Fixed $\lambda$, consider any family of solutions $\{u_{a,\lambda}\}_{a\in(0,1]}$ 
of \eqref{eq:eq} such that  
$$\mathcal I_{a,\lambda}(u_{a,\lambda}) < \frac13 S^{\frac32} - \lambda^{\frac{2}{1-p}} |K|_{\frac{2}{1-p}}^{\frac{2}{1-p}}\mathfrak a_{p}$$
where $\mathfrak a_{p}$ is the same constant as in Theorem \ref{th:a>0}.

Then as $a$ tends to $ 0$, the family $\{u_{a,\lambda}\}_{a\in(0,1]}$ is convergent,
up to subsequence, to a solution of  the Schr\"odinger-Poisson equation \eqref{eq:eqSP}
(with the same value of $\lambda$), denoted with $u_{0,\lambda}$. 

Moreover
\begin{eqnarray}
&& \lim_{a\to 0} \mathcal I_{0,\lambda}(u_{a,\lambda})=\lim_{a\to 0} \mathcal I_{a,\lambda}(u_{a,\lambda}) 
= \lim_{a\to 0} \mathcal I_{a,\lambda}(u_{0,\lambda})=
  \mathcal I_{0,\lambda}(u_{0,\lambda}), \label{eq:Ia0}\\
&&\lim_{a\to 0} \mathcal I'_{a,\lambda}(u_{0,\lambda})=
 \lim_{a\to 0} \mathcal I'_{0,\lambda}(u_{a,\lambda})= 0 \quad \text{ in } H^{-1}(\mathbb R^{3}).
 \label{eq:I'a0}
\end{eqnarray}
\end{theorem}
In particular  the convergences in \eqref{eq:new} and \eqref{eq:I'a0} say that,
as one parameter goes to zero,  solutions of 
\eqref{eq:eq} are ``almost'' solutions of \eqref{eq:eqSP} 
(for the same fixed value of the other parameter), and viceversa.

Note that, for the minimax solutions $u_{a,\lambda,n}$ and $u_{0,\lambda,n}$, with fixed $\lambda$ and $n$, 
found in Theorem \ref{vegeta} and Theorem \ref{th:a=0},  Theorem \ref{th:convina} does not say that 
$$\lim_{a\to 0} u_{a,\lambda,n}  = u_{0,\lambda,n} \quad \text{ in } H^{1}(\mathbb R^{3})$$
but it just states that the limit is a solution (and not necessarily the corresponding minimax solution $u_{0,\lambda,n}$).

However this is what happens for the ground state solutions, namely when $n=1$.
\begin{theorem}\label{th:GS1}
Assume \eqref{Q} and \eqref{W}.
If $\{u_{a,\lambda,1}\}_{a\in(0,1]}$ is the family of ground state solutions of \eqref{eq:eq}, then
$$ \lim_{a\to 0} u_{a,\lambda,1} = u_{0,\lambda,1} \quad \text{ in }\ H^{1}(\mathbb R^{3}),$$
where $u_{0,\lambda,1}$ is  a ground state solution of \eqref{eq:eqSP}.
\end{theorem}


\subsubsection*{Notations}
As a matter of notations, $B_{r}(a)$ denotes the ball centered in $a\in \mathbb R^{3}$ with radius $r>0$, and $B^{c}_{r}(a)$
its complementary. If $a=0\in \mathbb R^{3}$ we simply write $B_{r}$ and $B_{r}^{c}$.

The Lebesgue measure $dx$  in the integrals will be always omitted.


Moreover $c_{i}$ are positive and generic constant whose value is not really important and they can even change from
line to line.

Other notations will be introduced as soon as we need.

\subsubsection*{Organization of the paper}
The paper is organized as follows. In the next Section \ref{sec:PS}
we prove the Palais-Smale condition for the energy functional $\mathcal I_{a,\lambda}$. We will see that the range 
of validity of the Palais-Smale condition depends on $\lambda$.
In Section \ref{sec:minmax} a general minimax argument is implemented in order to understand 
the right geometry of the functional and then apply general theorems in Critical Point Theory.
In the final Section \ref{sec:Proofs} the proof of the results is given.

\medskip

In the remaining of the paper our hypothesis \eqref{Q} and \eqref{W} are implicitly assumed, without any
other reference.

\section{The Palais-Smale condition }\label{sec:PS}
In this and the next sections the results will be true for every $a\geq0$
which then has to be considered fixed; of course if $a=0$ we are dealing with the 
Schr\"odinger-Poisson system, if $a>0$ with the Schr\"odinger-Bopp-Podolsky one.
Moreover $\phi_{u}^{a}\in D^{1,2}$ or $\mathcal D$ accordingly.


Let us first recall a useful and fundamental condition in order to apply variational methods.
We say that a sequence $\{u_n\}_{n\in \mathbb N}\subset H^{1}(\mathbb{R}^3)$ is  a Palais-Smale sequence at the level $c\in\mathbb R$
for $\mathcal{I}_{a,\lambda}$, for short $(PS)_c$ sequence, if 
\begin{equation}\label{eq:PS}
\mathcal{I}_{a,\lambda}(u_n)\rightarrow c\quad\text{ and } \quad \mathcal{I}'_{a,\lambda}(u_n)\rightarrow 0.
\end{equation}
The  functional $\mathcal{I}_{a,\lambda}$ is said to satisfy the $(PS)_c$ condition if any $(PS)_c$ sequence contains a convergent subsequence,
and in this case the limit is a critical point.

\medskip

The aim  is to show that $\mathcal I_{a,\lambda}$ satisfies the $(PS)_{c}$ condition,
at least at suitable levels $c$.

So  let $a,\lambda$ be fixed and, in all this Section, let $\{u_n\}_{n\in \mathbb N}$ be a sequence satisfying  \eqref{eq:PS}
(of course the sequence depends also on $a$ and $\lambda$ but we omit these dependence).


\begin{lemma}\label{androide1}
    The $(PS)_{c}$ sequence  $\{u_n\}_{n\in \mathbb N}$ is bounded in $H^{1}(\mathbb{R}^3)$.
\end{lemma}
\begin{proof}
Indeed, since $p\in(0,1)$, the result follows by the estimates
\begin{eqnarray*}
   c+\varepsilon_n+\varepsilon_n \|u_{n}\|&=&\mathcal{I}_{a,\lambda}(u_n)-\frac{1}{6}\mathcal{I}'_{a,\lambda}(u_n)[u_n] \\
   &=& \frac{1}{3}\| u\|^{2}   +\frac{1}{12}\int_{\mathbb{R}^3} Q(x)\phi_{u_{n}}^{a} u_{n}^{2}
    - \lambda\left(\frac{1}{p+1}-\frac{1}{6} \right) \int_{\mathbb{R}^3} K(x)\vert u_n\vert^{p+1}\\
    &\geq& \frac{1}{3}\Vert u_n\Vert^{2}  -\lambda \frac{5-p}{6(p+1)} \vert K\vert_{\frac{2}{1-p}}\vert u_n\vert^{p+1}_{2}\\
&\geq&\frac{1}{3}\Vert u_n\Vert^{2} -\lambda\frac{5-p}{6(p+1)} |K|_{\frac{2}{1-p}}\Vert u_n\Vert^{p+1},
\end{eqnarray*}
where $\varepsilon_{n}\to 0$.
\end{proof}

\begin{remark}
Note that the bound does not depends on $a\geq0$ and moreover  it is uniform in $\lambda$
(so does not depends on the functional)  for $\lambda$ in a bounded set.
In particular, we can say that, given $\ell>0$ there is $C_{\ell}>0$ such that for every $\lambda\in (0,\ell)$
and  $a\geq0$, any $(PS)_{c}$ sequence for $\mathcal I_{a,\lambda}$ is bounded by $C_{\ell}$. 

Observe also that  there cannot be  Palais-Smale sequences at arbitrary negative values $c$ for the functional.
\end{remark}

Hence, 
 without loss of generality, we can assume that  
\begin{eqnarray*}\label{androide5}
\nonumber &u_n \rightharpoonup u & \text{in} \ H^{1}(\mathbb{R}^3),\\
 & u_n \rightarrow  u & \text{in} \ L^{p}_{loc}(\mathbb{R}^3) \ \ \text{for} \ 1\leq p<6,\\
\nonumber & u_n(x) \rightarrow u(x) & \text{a.e  in} \ \mathbb{R}^3.
\end{eqnarray*}
In the remaining of this Section $u$ will be always the weak limit of the $(PS)_{c}$
sequence $\{ u_{n}\}_{n\in \mathbb N}$.

By Lions' second concentration compactness lemma in \cite{PLions}, there exist an at most countable index set $\mathcal{J}$, a sequence of points $\{x_j\}_{j\in \mathcal{J}} \subset \mathbb{R}^3,$ and two sets  $\{\mu_j\}_{j\in \mathcal{J}}, \{\zeta\}_{j\in \mathcal{J}} \subset {\mathbb{R}^{+}}$ such that,
in the weak sense of measures, 
 \begin{equation}\label{androide3}
   \nonumber \vert \nabla u_n\vert^{2}\rightharpoonup d\mu \geq  \vert \nabla u\vert^{2}+\sum_{j\in \mathcal{J}} \mu_j \delta_{x_j},\qquad
   u_n^{6}\rightharpoonup d\zeta =  u^{6}+\sum_{j\in \mathcal{J}}\zeta_j \delta_{x_j},
    \end{equation}
    where $\delta_{x_j} $ is the Dirac mass at $x_j\in \mathbb{R}^3$,    and
\begin{equation}\label{eq:munu}
S\zeta_j^{\frac{1}{3}} \leq  \mu_j.
\end{equation}
Here $S$  is the best Sobolev constant, so
    $$ S\vert w \vert^{2}_{6}\leq  | \nabla w |_{2}^{2}, \quad \forall w\in H^{1}(\mathbb{R}^3).$$
    
In particular, if $\varepsilon>0$ is small enough, and
$\varphi_{\varepsilon} \in C^{\infty}_{c}(\mathbb{R}^3;[0,1])$  is such that 
    \begin{equation}\label{eq:bump}
         \varphi_{\varepsilon}  \equiv  1  \text{ on }  B_{\varepsilon}(x_{j_{0}}) \ \ \text{ and } \ \ 
        \varphi_{\varepsilon}  \equiv  0   \text{ on }  B^{c}_{2\varepsilon}(x_{j_{0}}), 
    \end{equation}
    then
     \begin{eqnarray*}\label{var}
\lim_{n\to\infty} \displaystyle  \int_{\mathbb{R}^3} \vert \nabla u_n \vert^{2} \varphi_{\varepsilon}  
 &= &\lim_{n\to\infty}\int_{B_{2\varepsilon}(x_{j_{0}})} \vert \nabla u_n \vert^{2} \varphi_{\varepsilon} 
 = \int_{B_{2\varepsilon}(x_{j_{0}})} \varphi_{\varepsilon} d\mu\\
 &\geq & \int_{B_{2\varepsilon}(x_{j_{0}})} \varphi_{\varepsilon}  \vert \nabla u\vert^{2}+ \int_{B_{2\varepsilon}(x_{j_{0}})} \varphi_{\varepsilon} \sum_{j\in \mathcal{J}} \mu_j \delta_{x_j}.
  \end{eqnarray*}  
   We conclude that  
\begin{eqnarray}
    \lim_{\varepsilon \rightarrow 0} \lim_{n\rightarrow\infty}\int_{\mathbb{R}^3} \vert \nabla u_n \vert^{2} \varphi_{\varepsilon}
     &\geq & 
     \lim_{\varepsilon \rightarrow 0} \left( \int_{B_{2\varepsilon}(x_{j_{0}})} \varphi_{\varepsilon} \vert \nabla u \vert^{2} + 
 \int_{B_{2\varepsilon}(x_{j_{0}})} \varphi_{\varepsilon} 
\sum_{j\in \mathcal{J}} \mu_j \delta_{x_j}        \nonumber \right)\\
          &=&\mu_{{j}_{0}} 
          \varphi_{\varepsilon}(x_{j_0})           \nonumber\\
          &=& \mu_{{j}_{0}} \label{eq:mu}.
\end{eqnarray}
 
 Analogously,
 \begin{eqnarray*}
    \lim_{n\rightarrow \infty}\int_{\mathbb{R}^3}|u_n|^{6}\varphi_{\varepsilon}&=& \lim_{n\rightarrow \infty}\int_{B_{2\varepsilon}(x_{j_{0}})}|u_n|^{6}\varphi_{\varepsilon}
 = \int_{B_{2\varepsilon}(x_{j_{0}})}\varphi_{\varepsilon} d\zeta\\
    &=& \int_{B_{2\varepsilon}(x_{j_{0}})} |u|^{6}\varphi_{\varepsilon} +  \int_{B_{2\varepsilon}(x_{j_{0}})} \varphi_{\varepsilon}
    \sum_{j\in \mathcal{J}}\zeta_j \delta_{x_j} ,
 \end{eqnarray*} 
 and we conclude that
\begin{equation}\label{eq:6}
     \lim_{\varepsilon \rightarrow 0} \lim_{n\rightarrow\infty} \int_{\mathbb{R}^3}|u_n|^{6}\varphi_{\varepsilon}=\zeta_{j_0}.
\end{equation}

    \begin{lemma}\label{lem:nu}
For each $j\in \mathcal J$ it is $\zeta_{j}\geq \mu_{j}.$
    \end{lemma}
    \begin{proof}
    Fixed $j_{0}\in \mathcal J$ and $\varepsilon>0$, define   a cut-off function 
    $\varphi_{\varepsilon} \in C^{\infty}_{c}(\mathbb{R}^3;[0,1])$  such that 
    \begin{equation}\label{phi}
         \varphi_{\varepsilon}  \equiv  1  \text{ on }  B_{\varepsilon}(x_{j_{0}}), \ \  
        \varphi_{\varepsilon}  \equiv  0   \text{ on }  B^{c}_{2\varepsilon}(x_{j_{0}}),  \ \
        \vert \nabla\varphi_{\varepsilon} \vert \leq {2}/{\varepsilon}.
    \end{equation}
Now,
\begin{eqnarray}\label{estimate}
  \mathcal{I}'_{a,\lambda}(u_n)[\varphi_{\varepsilon} u_n] 
    &=& \int_{\mathbb{R}^3} (\nabla u_n \nabla(\varphi_{\varepsilon} u_n)+  u_n^{2} \varphi_{\varepsilon}) +  \int_{\mathbb{R}^3} Q(x) \phi^{a}_{u_n} u_n^{2} \varphi_{\varepsilon} \nonumber \\ 
    &&-\int_{\mathbb{R}^3} \vert u_n\vert^6 \varphi_{\varepsilon} 
- \lambda \int_{\mathbb{R}^3} K(x)\vert u_n \vert^{p+1} \varphi_{\varepsilon} \nonumber \\
 &  = &  \int_{\mathbb{R}^3} u_n \nabla u_n \nabla \varphi_{\varepsilon}+  \int_{\mathbb{R}^3} \vert \nabla u_n\vert^{2}\varphi_{\varepsilon}+  \int_{\mathbb{R}^3}  u_{n}^{2} \varphi_{\varepsilon}  +  \int_{\mathbb{R}^3} Q(x) \phi^{a}_{u_n} u_n^{2} \varphi_{\varepsilon} \\
&& -\int_{\mathbb{R}^3} \vert u_n\vert^6 \varphi_{\varepsilon} - \lambda \int_{\mathbb{R}^3} K(x)\vert u_n \vert^{p+1} \varphi_{\varepsilon}\nonumber
\end{eqnarray}
and since the sequence $\{u_n\}_{n\in \mathbb N}$ is  bounded, the sequence  $\{\varphi_{\varepsilon} u_n \}_{n\in \mathbb N}$ is  bounded too in $H^{1}(\mathbb{R}^3)$;
 thus 
 \begin{equation}\label{eq:zero}
 \lim_{\varepsilon\to 0}\lim_{n\to\infty}       \mathcal{I}'_{a,\lambda}(u_n)[\varphi_{\varepsilon} u_n] 
=0.
 \end{equation}

Let us estimate  each term in the right hand side of  \eqref{estimate} by taking first the limit as $n\to\infty$ and then as $\varepsilon\to 0$.

By using  the H\"older inequality,   since $|u_n \nabla\varphi_{\varepsilon}| \in  L^{2}(\mathbb{R}^{3})$, we get
     \begin{eqnarray*}\label{girasol}
    0\leq \int_{\mathbb{R}^3} \vert u_n \nabla u_n \nabla \varphi_{\varepsilon}\vert &=&
    \int_{B_{2\varepsilon}(x_{j_{0}})}  \vert u_n \nabla u_n \nabla \varphi_{\varepsilon}\vert\\
     &\leq & \left(\int_{B_{2\varepsilon}(x_{j_{0}})}  \vert u_n \nabla \varphi_{\varepsilon}\vert^{2}\right)^{\frac12}\left(\int_{B_{2\varepsilon}(x_{j_{0}})}  \vert\nabla u_n \vert^{2}\right)^{\frac12} \\
     &\leq& \left(\int_{B_{2\varepsilon}(x_{j_{0}})}  \vert u_n \nabla \varphi_{\varepsilon}\vert^{2}\right)^{\frac12} \|u_{n}\|^{2}\\
     &\leq& c_{1} \left(\int_{B_{2\varepsilon}(x_{j_{0}})}  \vert u_n \nabla \varphi_{\varepsilon}\vert^{2}\right)^{\frac12}.
 \end{eqnarray*}
 Since $u_n \rightarrow u $ in $L^{2}(B_{2\varepsilon}(x_{j_{0}}))$, then by the Dominated Convergence Theorem and the H\"older inequality we get
 \begin{eqnarray*}
 0\leq \lim_{n\rightarrow\infty}
     \int_{\mathbb{R}^3} \vert u_n \nabla u_n \nabla \varphi_{\varepsilon}\vert &\leq&
     c_{1}\left(\int_{B_{2\varepsilon}(x_{j_{0}})}  \vert u\nabla \varphi_{\varepsilon}\vert^{2}\right)^{\frac12}\\
     &\leq& c_{1}
  \left(\int_{B_{2\varepsilon}(x_{j_{0}})}  |u|^{6}\right)^{\frac16}\left(\int_{B_{2\varepsilon}(x_{j_{0}})}  \vert\nabla \varphi_{\varepsilon} \vert^{3}\right)^{\frac13}\\
     &\leq & c_{2}\left(\int_{B_{2\varepsilon}(x_{j_{0}})}  | u|^{6}\right)^{\frac16},
 \end{eqnarray*}
 from which
 \begin{equation}\label{eq:gradgrad}
      \lim_{\varepsilon \rightarrow 0} \lim_{n\rightarrow\infty}
     \int_{\mathbb{R}^3} \vert u_n \nabla u_n \nabla \varphi_{\varepsilon}\vert =0.
 \end{equation}

The second and fifth terms have been already estimated in \eqref{eq:mu} and \eqref{eq:6}.

For the third term,
by (\ref{phi}) and since $u_n \rightarrow  u$ in  $L^{p}_{loc}(\mathbb{R}^3)$ for $1\leq p<6$, 
 it follows that
$$
   0\leq    \lim_{n\rightarrow \infty}\int_{\mathbb{R}^3} u_n^{2} \varphi_{\varepsilon} 
=   \lim_{n\to\infty}\int_{B_{2\varepsilon}(x_{j_{0}})} u_n^{2} \varphi_{\varepsilon}  \leq \lim_{n\rightarrow \infty}\int_{B_{2\varepsilon}(x_{j_{0}})} u_n^{2}=\int_{B_{2\varepsilon}(x_{j_{0}})} u^2,
$$
and consequently
\begin{eqnarray}\label{eq:quad}
  \lim_{\varepsilon\rightarrow 0}\lim_{n\rightarrow \infty}  \int_{\mathbb{R}^3} u_n^{2} \varphi_{\varepsilon} = 0.
\end{eqnarray}

The fourth term is easily estimated taking into account
\emph{vii-c)} 
 of Lemma    \ref{lemmaphi}, in fact
\begin{eqnarray*}
    \lim_{n\rightarrow \infty}\int_{\mathbb{R}^3} Q(x)\phi_{u_{n}}^{a}u_n^{2}\varphi_{\varepsilon}
    = \int_{\mathbb R^{3}} Q(x)\phi_{u}^{a}u^{2}\varphi_{\varepsilon}= \int_{B_{2\varepsilon}(x_{j_{0}})} Q(x)\phi_{u}^{a}u^{2}
\end{eqnarray*}
therefore, we conclude that
\begin{equation}\label{eq:Q}
    \lim_{\varepsilon\rightarrow 0}\lim_{n\rightarrow \infty} \int_{\mathbb{R}^3} Q(x)\phi_{u_{n}}^{a}u_n^{2}\varphi_{\varepsilon}=0.
\end{equation}

Finally, using the assumption  (\ref{W}) and  $ \vert u_n \vert^{p+1} \in L^{2/(p+1)}(\mathbb{R}^3)$ with $0<p<1 $, applying  the H\"older inequality, it follows that
\begin{eqnarray*}
0\leq \lim_{n\to\infty}\int_{\mathbb{R}^3}|K(x)|\vert u_n\vert ^{p+1}\varphi_{\varepsilon} &=&
  \lim_{n\to\infty} \int_{B_{2\varepsilon}(x_{j_{0}})}|K(x)|\vert u_n\vert ^{p+1}\varphi_{\varepsilon}\\
  &=& \int_{B_{2\varepsilon}(x_{j_{0}})}|K(x)|\vert u\vert ^{p+1}\varphi_{\varepsilon}\\
  &\leq&  \int_{B_{2\varepsilon}(x_{j_{0}})}|K(x)|\vert u\vert ^{p+1},
\end{eqnarray*}
from which
\begin{equation}\label{eq:ultimo}
     \lim_{\varepsilon\rightarrow 0}\lim_{n\rightarrow \infty} \int_{\mathbb{R}^3}K(x)\vert u_n\vert ^{p+1}\varphi_{\varepsilon}=0.
\end{equation}
From \eqref{estimate} and \eqref{eq:zero}, making use of 
\eqref{eq:mu}, \eqref{eq:6}, \eqref{eq:gradgrad}-\eqref{eq:ultimo},
we get
$     \zeta_{j_{0}}\geq \mu_{j_{0}}.$
 %
%
%
%
 \end{proof}
 %
 %
 %
 %
 

 \begin{corollary}\label{cor:zeta}
     For every $j\in \mathcal J$,
    $$\text{either } \ \zeta_{j} = 0 \ \text{ or }\ \zeta_{j}\geq S^{\frac32}.$$
    In particular, if $\zeta_{j}\geq S^{\frac32}$ then
$$ \frac{1}{4}\mu_{j} +\frac{1}{12}\zeta_{j} \geq \frac{1}{3}S^{\frac32}.$$
 \end{corollary}
\begin{proof}
By Lemma \ref{lem:nu} and \eqref{eq:munu} it is
$\zeta_{j}\geq \mu_{j}\geq S\zeta_{j}^{\frac13}$ and the first conclusion follows.
In particular, if $\zeta_{j}\geq S^{\frac32}$ we infer
 \begin{eqnarray*}\label{eq:fund}
 \frac{1}{4}\mu_{j} +\frac{1}{12}\zeta_{j} \geq \frac{1}{4}\mu_{j} + \frac{1}{12}\mu_{j}
 =\frac13 \mu_{j}
 \geq\frac13 S\zeta_{j}^{\frac13}
\geq\frac13 S S^{\frac12} = \frac{1}{3}S^{\frac32}
 \end{eqnarray*}
concluding the proof.
\end{proof}

Since $\{u_{n}\}_{n\in \mathbb N}$ is  a $(PS)_{c}$  sequence, we know that
\begin{eqnarray*}
c&=& \lim_{n\to\infty} \left( \mathcal{I}_{a,\lambda}(u_n)-\frac{1}{4} \mathcal{I}_{a,\lambda}'(u_n)[u_n] \right) \nonumber\\
&=& \lim_{n\to\infty}
\left( \frac{1}{4}\int_{\mathbb{R}^3}\vert \nabla u_n\vert^{2} +\frac{1}{4}\int_{\mathbb{R}^3} \vert u_n\vert ^2+ \frac{1}{12}\int_{\mathbb{R}^3} \vert u_n\vert ^{6} -\lambda \frac{3-p}{4(p+1)}\int_{\mathbb{R}^3}K(x) \vert u_n \vert^{p+1} \right)\nonumber \\
 &\geq& \lim_{n\to\infty} \left(\frac{1}{4}\int_{\mathbb{R}^3}\vert \nabla u_n\vert^{2} 
 + \frac{1}{12}\int_{\mathbb{R}^3} \vert u_n\vert ^{6}  +  \theta_{\lambda}(u_{n}) \right)
\end{eqnarray*}
where 
\begin{equation}\label{eq:t}
\theta_{\lambda}(u_{n}) = 
\frac14 |u_{n}|_{2}^{2} - \lambda\frac{3-p}{4(p+1)}|K|_{\frac{2}{1-p}}|u_{n}|_{2}^{p+1}.
\end{equation}
 Then it is natural to consider the  function $g_{\lambda}(t) = c_{1}t^{2} - \lambda c_{2}t^{p+1}$, and  straightforward computations
 show that
 \begin{equation*}
 g_{\lambda}(t) \geq \min_{t\geq0} g_{\lambda}(t) = -  \lambda^{\frac{2}{1-p}}(1-p)\left(\frac{c_{2}}{2} \right)^{\frac{2}{1-p}} \left( \frac{p+1}{c_{1}}\right)^{\frac{1+p}{1-p}}.
 \end{equation*} 
Applying this to our function \eqref{eq:t} we see that, defining
 \begin{equation*}
  \mathfrak a_{p} :=  \frac{1-p}{1+p} \left(\frac{(3-p)^{2}}{4^{2-p}} \right)^{\frac{1}{1-p}}>0,
 \end{equation*}
 we have
 $\theta_{\lambda}(u_{n}) \geq  -\lambda^{\frac{2}{1-p}} |K|_{\frac{2}{1-p}}^{\frac{2}{1-p}}
  \mathfrak a_{p}.
  $
 From now on $ \mathfrak a_{p}$ is the constant given above. 

Now, let  $j_{0}\in \mathcal J$ and 
 take  $\varphi_{\varepsilon}\in C_{c}^{\infty}(\mathbb R^{3};[0,1])$ a bump centered in 
$x_{j_{0}}$ (hence satisfying  \eqref{eq:bump}). We see that
 \begin{eqnarray*}
 c&\geq&   \lim_{n\to\infty} \left(\frac{1}{4}\int_{\mathbb{R}^3}\vert \nabla u_n\vert^{2} \varphi_{\varepsilon} + 
   \frac{1}{12}\int_{\mathbb{R}^3} \vert u_n\vert ^{6}  \varphi_{\varepsilon}
 \right) - \lambda^{\frac{2}{1-p}}  |K|_{\frac{2}{1-p}}^{\frac{2}{1-p}} \mathfrak a_{p} \\ 
 &\geq& \frac14 \int_{\mathbb R^{3}}|\nabla u|^{2}\varphi_{\varepsilon}  +\frac14\mu_{j_{0}}+
 \frac{1}{12}\int_{\mathbb R^{3}}|u|^{6}\varphi_{\varepsilon} + \frac{1}{12}\zeta_{j_{0}}-\lambda^{\frac{2}{1-p}}
 |K|_{\frac{2}{1-p}}^{\frac{2}{1-p}} \mathfrak a_{p} .
 \end{eqnarray*}
Passing to the limit as $\varepsilon\to \infty$,  
we find
\begin{equation}\label{eq:stimac}
c\geq \frac14 \int_{\mathbb R^{3}}|\nabla u|^{2} +
 \frac{1}{12}\int_{\mathbb R^{3}}|u|^{6}+ \frac14\mu_{j_{0}}+\frac{1}{12}\zeta_{j_{0}}-\lambda^{\frac{2}{1-p}}   |K|_{\frac{2}{1-p}}^{\frac{2}{1-p}}\mathfrak a_{p} 
\end{equation}
and we immediately deduce  the next  result.
 \begin{lemma}\label{lem:zero}
If
 $$c<\frac{1}{3}S^{3/2} - \lambda^{\frac{2}{1-p}}  |K|_{\frac{2}{1-p}}^{\frac{2}{1-p}} \mathfrak a_{p}$$ 
  then, for every $j\in \mathcal J$, it is $\zeta_{j} = 0$.
In particular, by Lemma \ref{lem:nu}, it is   $\mu_{j}=0$.
 \end{lemma}
\begin{proof}
 Using Corollary \ref{cor:zeta},  assume by contradiction that for some $j_{0}$
 it is $    \zeta_{j_{0}}\geq S^{\frac32}$.
 From \eqref{eq:stimac} it follows
$$
c\geq\frac14\mu_{j_{0}}+\frac{1}{12}\zeta_{j_{0}}-\lambda^{\frac{2}{1-p}}  |K|_{\frac{2}{1-p}}^{\frac{2}{1-p}}\mathfrak a_{p}
 \geq  \frac13 S^{\frac32} -\lambda^{\frac{2}{1-p}}  |K|_{\frac{2}{1-p}}^{\frac{2}{1-p}}\mathfrak a_{p}
$$
which is absurd.
\end{proof}
The above result will be useful to show that any functional satisfies
the Palais-Smale condition below a certain  level which is varying with the functional.

Then by the Lemma it follows that, in the sense of measures, actually we have
\begin{eqnarray}\label{eq:misura}
|u_{n}|^{6} \rightharpoonup |u|^{6} \ \ \ \text{ and }\ \ \ 
|\nabla u_{n}|^{2}  \rightharpoonup \mu \geq |\nabla u|^{2}.
\end{eqnarray}

The next step is to address the convergence in  $L^{6}(\mathbb R^{3})$
of $\{u_{n}\}_{n\in \mathbb N}$.
Let us start with the  local convergence.

\begin{lemma}\label{lem:loc}
If  
 $$c<\frac{1}{3}S^{\frac{3}{2}} - \lambda^{\frac{2}{1-p}}  |K|_{\frac{2}{1-p}}^{\frac{2}{1-p}}\mathfrak a_{p},  $$
 for every $R>0$ we have
$$      \lim_{n\to\infty} \int_{B_{R}} |u_{n}|^{6} =  \int_{B_{R}}|u|^{6}.$$
\end{lemma}
\begin{proof}
This is a result in Measure Theory, however for the reader convenience we give the details.
  Fix $R>0$ and for every $\varepsilon>0$  consider $\varphi_{\varepsilon}\in C^{\infty}_{c}(\mathbb R^{3};[0,1])$
  such that 
  $$\varphi_{\varepsilon}\equiv 1 \ \text{ on } B_{R-\varepsilon}, \quad \varphi_{\varepsilon}\equiv0 \ \text{ on }
  B_{R}^{c}.$$
  Then by \eqref{eq:misura}
  \begin{equation*}
  \liminf_{n\to\infty} \int_{B_{R}} |u_{n}|^{6} \geq   \liminf_{n\to\infty} \int_{\mathbb R^{3}} 
  \varphi_{\varepsilon} |u_{n}|^{6}  =   \int_{\mathbb R^{3}}   \varphi_{\varepsilon}|u|^{6}
  \end{equation*}
  and by the Lebesgue's Theorem it holds, as $\varepsilon\to0$,
  \begin{equation}\label{eq:inf}
    \liminf_{n\to\infty} \int_{B_{R}} |u_{n}|^{6} \geq  \int_{B_{R}}|u|^{6}.
  \end{equation}
 
 Similarly, by taking $\varphi_{\varepsilon}\in C^{\infty}_{c}(\mathbb R^{3};[0,1])$
  such that 
  $$\varphi_{\varepsilon}\equiv 1 \ \text{ on } B_{R}, \quad \varphi_{\varepsilon}\equiv0 \ \text{ on }
  B_{R+\varepsilon}^{c},$$ 
  we infer
  $$\limsup_{n\to\infty} \int_{B_{R}} |u_{n}|^{6}\leq \limsup_{n\to\infty} \int_{\mathbb R^{3}} \varphi_{\varepsilon}|u_{n}|^{6}= \int_{\mathbb R^{3}} \varphi_{\varepsilon}|u|^{6}$$
  and passing to the limit as $\varepsilon\to 0 $ we get
\begin{equation}\label{eq:sup}
\limsup_{n\to\infty} \int_{B_{R}} |u_{n}|^{6}\leq \int_{B_{R}}|u|^{6}.
\end{equation}
The conclusion then follows by \eqref{eq:inf} and \eqref{eq:sup}.
\end{proof}

Finally  we can prove the fundamental result.

\begin{proposition}\label{prop:6}
If 
  $$c<\frac{1}{3}S^{\frac{3}{2}} - \lambda^{\frac{2}{1-p}}  |K|_{\frac{2}{1-p}}^{\frac{2}{1-p}}\mathfrak a_{p},  $$
 it is 
$$\lim_{n\to\infty} |u_{n}|_{6}^{6} = |u|_{6}^{6}$$
and then $u_{n}\to u$ in $L^{6}(\mathbb R^{3})$.
\end{proposition}
\begin{proof}

By the Fatou's Lemma  it is
\begin{eqnarray*}
   \int_{\mathbb{R}^3} |u|^{6} 
    \leq \liminf _{n \rightarrow \infty} \int_{\mathbb{R}^3} |u_n|^{6}
\end{eqnarray*}  
and so it is sufficient to prove that 
\begin{equation}\label{eq:sup6}
     \limsup _{n \rightarrow \infty} \int_{\mathbb{R}^3} |u_n|^{6}\leq   \int_{\mathbb{R}^3} |u|^{6} .\end{equation}
%
%
  Note  first that, in virtue of Lemma \ref{lem:loc} 
  \begin{eqnarray}\label{eq:bho}
  \limsup_{n\to\infty}\int_{\mathbb R^{3}} |u_{n}|^{6} &= &
    \limsup_{n\to\infty} \left(\int_{B_{R}^{c}} |u_{n}|^{6}  +
\int_{B_{R}} |u_{n}|^{6} \right) \nonumber \\
&      =&  \limsup_{n\to\infty} \int_{B_{R}^{c}} |u_{n}|^{6}  + \int_{B_{R}}|u|^{6}.
  \end{eqnarray}
If we set
  $$\zeta_{\infty} := \lim_{R\to+\infty} \limsup_{n\to\infty}\int_{B_{R}^{c}} |u_{n}|^{6} $$
  by \eqref{eq:bho} we infer, as $R\to+\infty$,
  \begin{equation*}\label{eq:}
    \limsup_{n\to\infty}\int_{\mathbb R^{3}} |u_{n}|^{6} =  \zeta_{\infty} + \int_{\mathbb R^{3}}|u|^{6}.
  \end{equation*}
  So the proof is finished if we show that $\zeta_{\infty} = 0$, since in this case we would have \eqref{eq:sup6} (indeed with the equality).

Let us define also 
  \begin{equation*}\label{eq:}
  \mu_{\infty}:= \lim_{R\to+\infty}\limsup_{n\to\infty} \int_{B_{R}^{c}}|\nabla u_{n}|^{2}
  \end{equation*}
  and  as in \cite[Lemma 1.40]{MWill} we see that  
\begin{equation}\label{eq:zetamuinfty}
S\zeta_{\infty}^{\frac{1}{3}} \leq \mu_{\infty}.
\end{equation}    
  
  {\bf Claim: } $\zeta_{\infty} \geq\mu_{\infty}$.
  
\medskip

%
  Consider a cut-off function $\varphi_{R}\in C^{\infty}(\mathbb R^{3};[0,1])$ with
  \begin{equation}\label{eq:phiR}
  \varphi_{R}\equiv 0\ \text{ on } B_{R}, \quad \varphi_{R}\equiv1 \ \text{ on }
  B_{R+1}^{c}, \ \ |\nabla \varphi_{R}| \leq \frac2R
  \end{equation}
  Since $\{u_{n}\varphi_{R}\}_{n\in \mathbb N}$ is bounded in $H^{1}(\mathbb R^{3})$ we have
  $\lim_{n\to\infty}\mathcal I'_{a,\lambda}(u_{n})[u_{n}\varphi_{R}] = 0$, which means
  \begin{eqnarray}\label{estimate2}
      \mathcal{I}'_{a,\lambda}(u_n)[\varphi_{R} u_n ] 
    &=& \int_{\mathbb{R}^3} (\nabla u_n \nabla(\varphi_{R} u_n)+  u_n^{2} \varphi_{R}) +  \int_{\mathbb{R}^3} Q(x) \phi^{a}_{u_n} u_n^{2} \varphi_{R} \nonumber \\ 
    &&-\int_{\mathbb{R}^3} \vert u_n\vert^6 \varphi_{R} 
- \lambda \int_{\mathbb{R}^3} K(x)\vert u_n \vert^{p+1} \varphi_{R} \nonumber \\
 &  \geq &  \int_{\mathbb{R}^3} u_n \nabla u_n \nabla \varphi_{R}
 +  \int_{\mathbb{R}^3} \vert \nabla u_n\vert^{2}\varphi_{R} +   \int_{\mathbb{R}^3} Q(x) \phi^{a}_{u_n} u_n^{2} \varphi_{R} \nonumber
 \\
&& -\int_{\mathbb{R}^3} \vert u_n\vert^6 \varphi_{R} - \lambda \int_{\mathbb{R}^3} K(x)\vert u_n \vert^{p+1} \varphi_{R}.
\end{eqnarray}
Let us evaluate each term above by taking first the limsup as $n\to+\infty$ and then as $R\to+\infty$.

  By repeating the computations which led to \eqref{eq:gradgrad} we have 
$$ 
0\leq \lim_{n\rightarrow\infty}
     \int_{\mathbb{R}^3} \vert u_n \nabla u_n \nabla \varphi_{R}\vert 
    \leq
     c_{1}\left(\int_{B_{R+1}\setminus B_{R}}  \vert u\vert^{6}\right)^{\frac{1}{6}},
$$
 from which
 \begin{equation}\label{eq:gradgrad2}
      \lim_{R \rightarrow +\infty} \lim_{n\rightarrow\infty}
     \int_{\mathbb{R}^3} \vert u_n \nabla u_n \nabla \varphi_{R}\vert =0.
 \end{equation}

 For the second and fourth terms note that
 \begin{eqnarray*}
&& \int_{B^{c}_{R+1}} |\nabla u_{n}|^{2} \leq \int_{\mathbb R^{3}} |\nabla u_{n}|^{2}  \varphi_{R}\leq \int_{B^{c}_{R}} |\nabla u_{n}|^{2} \\
&& \int_{B^{c}_{R+1}} | u_{n}|^{6} \leq \int_{\mathbb R^{3}} | u_{n}|^{6}  \varphi_{R}\leq \int_{B^{c}_{R}} |u_{n}|^{6}
 \end{eqnarray*}
 so
 \begin{eqnarray}
&& \lim_{R\to+\infty}\limsup_{n\to\infty}\int_{\mathbb R^{3}} |\nabla u_{n}|^{2}  \varphi_{R} = \mu_{\infty}\label{eq:muinfty}, \\
&& \lim_{R\to+\infty}\limsup_{n\to\infty}\int_{\mathbb R^{3}} | u_{n}|^{6}  \varphi_{R} = \zeta_{\infty}\label{eq:zetainfty}.
 \end{eqnarray}
 
 The third term is estimates as 
 \begin{eqnarray*}
 \limsup_{n\to+\infty} \Big| \int_{\mathbb R^{3}} Q(x) \phi_{u_{n}}^{a} u_{n}^{2}\varphi_{R} \Big| &\leq &
  \limsup_{n\to+\infty} \int_{B_{R}^{c}} |Q(x) \phi_{u_{n}}^{a}| u_{n}^{2}\\
  &\leq &  \limsup_{n\to+\infty} |Q|_{L^{s}(B_{R}^{c})} 
   |\phi_{u_{n}}^{a}|_{L^{6}(B_{R}^{c})} |u_{n}^{2}|_{L^{\frac{6s}{5s-6}}(B_{R}^{c})}
  \\
  &\leq &  \limsup_{n\to+\infty}  c_{1}  |Q|_{L^{s}(B_{R}^{c})}
  \|\phi_{u_{n}}^{a}\|_{\mathcal D} \| u_{n}\|^{2}\\
  &\leq &c_{2} |Q|_{L^{s}(B_{R}^{c})} 
 \end{eqnarray*}
having used that  $\frac{12s}{5s-6}\in[2,6]$, \eqref{v:lemmaphi} of Lemma \ref{lemmaphi}
and that $\{u_{n}\}_{n\in \mathbb N}$ is bounded.  It follows that 
 \begin{equation}\label{eq:Q}
 \lim_{R\rightarrow +\infty}   \limsup _{n \rightarrow \infty} \int_{\mathbb{R}^3} Q(x) \phi_{u_{n}}^{a} u_{n}^{2}\varphi_{R} =0.
 \end{equation}
 
 Finally for the fifth term, using the boundedness of $\{u_{n}\}_{n\in \mathbb N}$ in $L^{2}(\mathbb R^{3})$,
 \begin{eqnarray*}
   \limsup _{n \rightarrow \infty}\Big| \int_{\mathbb{R}^3} K(x)\vert u_n \vert^{p+1}\varphi_{R} \Big|&\leq& \limsup _{n \rightarrow \infty}
   \int_{B^{c}_{R}} |K(x)|\vert u_n \vert^{p+1}\\
   &\leq & \limsup _{n \rightarrow \infty}\left(\int_{B^{c}_{R}}  \vert K(x)\vert^{\frac{2}{1-p}}\right)^{\frac{1-p}{2}} \left(\int_{B^{c}_{R}} \vert u_n \vert^2 \right)^{\frac{p+1}{2}}\\
&\leq& c_{3} \left(\int_{B^{c}_{R}}  \vert K(x)\vert^{\frac{2}{1-p}}\right)^{\frac{1-p}{2}}
\end{eqnarray*}
and so
\begin{equation}\label{eq:omega}
\lim_{R\rightarrow +\infty}   \limsup _{n \rightarrow \infty} \int_{\mathbb{R}^3} K(x)\vert u_n \vert^{p+1}\varphi_{R} =0.
\end{equation}

Summing up,  taking the limsup in $n$ in \eqref{estimate2}, and successively the limit as $R\to+\infty$, by means of \eqref{eq:gradgrad2}-\eqref{eq:omega} we deduce that $\zeta_{\infty}\geq\mu_{\infty}$ which proves the claim

\medskip

In virtue of \eqref{eq:zetamuinfty} it is 
$$\zeta_{\infty}=0 \quad\text{or}\quad\zeta_{\infty} \geq S^{\frac32}$$
and arguing as in Corollary \ref{cor:zeta} we infer
that if $\zeta_{\infty}\geq S^{\frac32}$ then

\begin{equation}\label{eq:combinfty}
\frac{1}{4}\mu_{\infty} +\frac{1}{12}\zeta_{\infty} \geq \frac{1}{3}S^{\frac32}.
\end{equation}


But then we can argue  as in the proof of Lemma \ref{lem:zero} and get a contradiction.
Hence $\zeta_{\infty} =0$ and the proof is concluded.
%
\end{proof}

We have the following convergences that will be useful to prove the compactness condition.
Note that the strong convergence in $L^{6}(\mathbb R^{3})$ of $\{u_{n}\}_{n\in \mathbb N}$ is not used in the proof.
Indeed the next result is true for any  sequence weakly convergent in $H^{1}(\mathbb R^{3})$.

\begin{lemma}\label{lem:final}
We have
\begin{eqnarray*}
&\displaystyle \int_{\mathbb R^{3}} K(x)|u_{n}|^{p+1} \longrightarrow \int_{\mathbb R^{3}}K(x)|u|^{p+1},&\\
&\displaystyle \int_{\mathbb R^{3}} K(x)|u_{n}|^{p-1}u_{n}u \longrightarrow\int_{\mathbb R^{3}} K(x)|u|^{p+1}.&
\end{eqnarray*}
\end{lemma}
\begin{proof}

Fixed $\varepsilon>0$ there is $R_{\varepsilon}>0$ such that
$|K|_{L^{\frac{2}{1-p}}(B_{R_{\varepsilon}}^{c})} \leq \varepsilon$. Then we have
\begin{eqnarray}\label{eq:K}
\int_{\mathbb R^{3}} |K(x)| \left| |u_{n}|^{p+1} - |u|^{p+1}\right| &\leq&
\int_{B_{R_{\varepsilon}}} |K(x)| \left| |u_{n}|^{p+1} - |u|^{p+1}\right| \nonumber\\
&+&  \varepsilon  \left| |u_{n}|^{p+1} - |u|^{p+1}\right|_{L^{\frac{2}{1+p}}(B_{R_{\varepsilon}}^{c}) } \nonumber \\
&\leq&\int_{B_{R_{\varepsilon}}} |K(x)| \left| |u_{n}|^{p+1} - |u|^{p+1}\right|  + \varepsilon \left( |u_{n}|_{2}^{p+1} +|u|^{p+1}_{2} \right) \nonumber\\
&\leq& \int_{B_{R_{\varepsilon}}} |K(x)| \left| |u_{n}|^{p+1} - |u|^{p+1}\right|  + \varepsilon c_{1}.
\end{eqnarray}
Setting $A_{M}:=\{x\in B_{R_{\varepsilon}} : |K(x)|>M \}$ we know that 
$\lim_{M\to+\infty}\textrm{meas} \,A_{M} = 0$. Therefore for $M$ large it is 
$|K|_{L^{\frac{2}{1-p}}(A_{M})} \leq \varepsilon$
and so we have, 
\begin{eqnarray*}
\int_{B_{R_{\varepsilon}}} |K(x)| \left| |u_{n}|^{p+1} - |u|^{p+1}\right| &\leq& 
\varepsilon \left| {|u_{n}|^{p+1} - |u|^{p+1}}\right|_{L^{\frac{2}{p+1}}(A_{M})}+M\int_{B_{R_{\varepsilon}}\setminus A_{M}}\left| |u_{n}|^{p+1} - |u|^{p+1}\right| \\ 
&\leq&\varepsilon (|u_{n}|^{p+1}_{2} + |u|^{p+1}_{2}) 
+ M \left| |u_{n}|^{p+1} - |u|^{p+1}\right|_{L^{1}(B_{R_{\varepsilon}})}\\
&\leq & \varepsilon c_{1}  +  o_{n}(1)
\end{eqnarray*}
due to the strong convergence in the ball. 
Inserting the above estimate in \eqref{eq:K} we get the   the first convergence due to the arbitrariness of $\varepsilon$.

\smallskip

The other  convergence is proved similarly: fixed $\varepsilon>0$ choose $R_{\varepsilon}>0$ 
so that $|K|_{L^{\frac{2}{1-p}}(B_{R_{\varepsilon}}^{c})} \leq \varepsilon$ and then
\begin{eqnarray*}
\int_{\mathbb R^{3}} |K(x)| \left| |u_{n}|^{p-1}u_{n} - |u|^{p}\right| |u| &\leq&
\int_{B_{R_{\varepsilon}}}  |K(x)| \left| |u_{n}|^{p-1}u_{n} - |u|^{p}\right| |u| \nonumber\\
&+& \varepsilon ||u_{n}|^{p-1} u _{n} - |u|^{p} |_{L^{\frac2p}(B_{R_{\varepsilon}}^{c})}  |u|_{L^{2}(B_{R_{\varepsilon}}^{c})}.
\end{eqnarray*}
The proof is concluded as before.
\end{proof}

Finally we are able to prove the Palais-Smale condition.

\begin{proposition}\label{cor:PS}
If 
$$c<\frac13 S^{\frac32}-\lambda^{\frac{2}{1-p}}|K|_{\frac{2}{1-p}}^{\frac{2}{1-p}}\mathfrak a_{p} $$
then the $(PS)_{c}$ sequence $\{u_{n}\}_{n\in \mathbb N}$
for the functional  $\mathcal I_{a,\lambda}$ is strongly convergent (up to subsequences) to $u$ in $H^{1}(\mathbb R^{3})$.
\end{proposition}
\begin{proof}

By using Lemma \ref{lem:final}, the strong convergence  $u_{n}\to u$ in $L^{6}(\mathbb R^{3})$ 
and the convergences in \emph{vii-b)} and \emph{vii-d)}  of Lemma \ref{lemmaphi},
we infer
\begin{eqnarray*}
0 \longleftarrow \mathcal I_{a,\lambda}(u_{n})[u] &=& \langle u_{n}, u \rangle + \int_{\mathbb R^{3}}Q(x) \phi^{a}_{u_{n}} u_{n} u -\int_{\mathbb R^{3}}|u_{n}|^{4} u_{n} u -\lambda \int_{\mathbb R^{3}} K(x) |u_{n}|^{p-1} u_{n} u\nonumber\\
&\longrightarrow& \|u\|^{2} + \int_{\mathbb R^{3}}Q(x) \phi^{a}_{u} u^{2} -\int_{\mathbb R^{3}}|u|^{6}-\lambda \int_{\mathbb R^{3}} K(x) |u|^{p+1}
\end{eqnarray*}
and
\begin{eqnarray*}
0 \longleftarrow \mathcal I_{a,\lambda}(u_{n})[u_{n}] &=& \|u_{n}\|^{2} + \int_{\mathbb R^{3}}Q(x) \phi^{a}_{u_{n}} u^{2}_{n} -\int_{\mathbb R^{3}}|u_{n}|^{6}-\lambda \int_{\mathbb R^{3}} K(x) |u_{n}|^{p+1}\\
&=& \|u_{n}\|^{2} + \int_{\mathbb R^{3}}Q(x) \phi^{a}_{u} u^{2} -\int_{\mathbb R^{3}}|u|^{6}-\lambda \int_{\mathbb R^{3}} K(x) |u|^{p+1} +o_{n}(1)
\end{eqnarray*}
Then, by subtracting the above relations we find $\|u_{n}\| \to \|u\|$, which implies the strong convergence of $\{u_{n}\}_{n\in \mathbb N}$ to $u$ completing the proof.
\end{proof}

\section{The minimax scheme}\label{sec:minmax}
In this section we use a topology tool, called Kransnoselki genus, to prove that  the functional
$\mathcal I_{a,\lambda}$ has the right geometric properties in order to apply
classical theorems in Critical Point Theory.
We first recall few general facts. Let $\Sigma$ be the class of closed and symmetric subset 
of an Hilbert space $E$,
$\gamma$ be the Krasnoselkii genus:
$$
\gamma: A\in \Sigma \longmapsto
\begin{cases}
\min\{n\in \mathbb N: \exists h\in C(A;\mathbb R^{n}\setminus \{0\}), h \text{ odd } \},\\
+\infty &\mbox{ if } \{\ldots\} =\emptyset, \\
0&\mbox{ if } A =\emptyset.
\end{cases}
$$

Let us recall the following properties of the Krasnoselskii genus that we will use in the following.
For more details see \cite{Rabinowitz}.
\begin{proposition}
Let $A,B\in \Sigma$. Then
 \begin{enumerate}
 \item if there is a continuous and odd function from $A$ to $B$, then $\gamma(A)\leq \gamma(B)$, \smallskip
 \item if $A\subset B$, then $\gamma(A)\leq \gamma(B)$, \smallskip
 \item if there is an odd homeomorphism between $A$ and $B$, then $\gamma(A)=\gamma(B)$, \smallskip
 \item if $\mathbb S^{N-1}$ is the unit sphere in $\mathbb R^{N}$, then $\gamma(\mathbb S^{N-1})=N$, \smallskip
 \item $\gamma(A\cup B)\leq \gamma(A)+\gamma(B)$, \smallskip
 \item if $\gamma(B)<+\infty$, then $\gamma(\overline{A\setminus B})\geq \gamma(A) - \gamma(B)$, \smallskip
 \item if $A$ is compact, then $\gamma(A)<+\infty$ and there is  
 $N_{\delta}(A)= \{ x\in E: d(x,A)\leq \delta\}$ with $\gamma(A)=\gamma(N_{\delta}(A))$,
 \item if $\gamma(A)\geq 2,$ then the set $A$ has infinitely many points.
 \end{enumerate}
\end{proposition}

Now, using the H\"{o}lder and Sobolev inequality, we find  
\begin{eqnarray*}
   \nonumber \mathcal{I}_{a,\lambda}(u)
   \nonumber &\geq &   \frac{1}{2}\int_{\mathbb{R}^3}( \vert \nabla u \vert^{2}+ u^{2})- \frac{1}{6}\int_{\mathbb{R}^3} \vert u\vert^{6}- \frac{\lambda}{p+1}\int_{\mathbb{R}^3} K(x)\vert u \vert^{p+1}\\
    & \geq & \frac{1}{2} \Vert u\Vert^{2} -\frac{1}{6} \vert u \vert^{6}_{6}-\frac{\lambda}{p+1} \vert K \vert_{\frac{2}{1-p}} \vert u \vert^{p+1}_{2}\\
        &\geq & \frac{1}{2}\Vert u\Vert^{2} -\frac{1}{6S^{3}} \Vert u\Vert^{6}-\frac{\lambda }{p+1}\vert K \vert_{\frac{2}{1-p}}\Vert u\Vert^{p+1} .
\end{eqnarray*}

As in \cite{AzP}, we consider the auxiliary function 
\begin{equation*}
        f_{\lambda}(t)=\frac{1}{2}t^{2}-\frac{1}{6S^{3}}t^{6}- \frac{\lambda }{p+1} \vert K \vert_{\frac{2}{1-p}}t^{p+1},
    \quad t>0.
\end{equation*}
Straightforward computations, being $p\in (0,1)$, show that, setting
\begin{equation}\label{eq:L1}
\Lambda_{1} : =  \frac{2(1+p)}{(5-p)|K|_{\frac{2}{1-p}} }\Big( \frac{3(1-p)S^{3}}{5-p} \Big)^{\frac{1-p}{4}} 
\end{equation} 
the function
$f_{\Lambda_{1}}$ has its unique maximum in $(0,+\infty)$ at the zero level.
Hence, for 
$\lambda\in(0,\Lambda_{1})$  the function $f_{\lambda}$  has 
 a unique  maximum at a strictly positive level, and has two zeroes:
let us call them $0<t_{\lambda,0}^{*}<t_{\lambda,0}^{**}$. 

Now define  a truncated functional  as follows:
\begin{equation*}
\mathcal{\widetilde{I}}_{a,\lambda}(u):= \frac{1}{2} \|u\|^{2}+ \frac{1}{4} \int_{\mathbb{R}^3}Q(x) \phi^{a}_{u}u^{2}dx- \frac{\tau(\|u\|)}{6} |u|_{6}^{6}-\frac{\lambda}{p+1}\int_{\mathbb{R}^3} K(x) \vert u\vert^{p+1}
\end{equation*}
 where 
  $ \tau \in C^{\infty}(\mathbb R^{3}; [0,1])$ is decreasing and  such that
$$
\tau(\Vert u \Vert) =
 \begin{cases}
   1, \ \ \mbox{ if } \| u \|\leq t_{\lambda,0}^{*}, \smallskip\\
    0, \ \  \mbox{ if } \| u \|\geq t_{\lambda,0}^{**},  
 \end{cases}
 $$
 and then,
 $$\mathcal{\widetilde{I}}_{a,\lambda}(u)\geq \widetilde{f_{\lambda}}(\Vert u \Vert), 
 \quad\text{with } \ 
      \widetilde{ f}_{\lambda}(t)=\frac{1}{2}t^{2}- \frac{\tau(t)}{6S^{3}}t^{6}- \frac{\lambda }{p+1} \vert K \vert_{\frac{2}{1-p}}t^{p+1}.
$$
Note that $\mathcal{\widetilde{I}}_{a,\lambda}\in C^{1}$, is bounded from below and 
 if $\mathcal{\widetilde{I}}_{a,\lambda}(u)<0$ then it is $\Vert u\Vert \leq t_{\lambda,0}^{*}$, implying
 $\mathcal{\widetilde{I}}_{a,\lambda}(u)=\mathcal{I}_{a,\lambda}(u)$.
Hence we are reduced to  find critical points for $\widetilde{ \mathcal I}_{a,\lambda}$ at negative energy level.

To this aim, let us introduce a minimax scheme.

Consider for $n\in \mathbb N, n\geq1, d\in \mathbb R$
 \begin{eqnarray*}\label{app7}
   &&  \nonumber \Sigma_{n}:=\{A \in\Sigma : \gamma(A)\geq n\},\\
    && \textrm K_{d}:=\{u \in H^{1}(\mathbb{R}^3): \mathcal{\widetilde{I}}_{a,\lambda}(u)=d, \ \mathcal{\widetilde{I}'}_{a,\lambda}(u)=0 \},\\
     &&\nonumber \mathcal{\widetilde{I}}_{a}^{d}:= \{u \in H^{1}(\mathbb{R}^3):\mathcal{\widetilde{I}}_{a,\lambda}(u)\leq d\},
 \end{eqnarray*}
 and define
 \begin{equation}\label{app8}
     b_{a,\lambda,n}:=\inf_{A\in \Sigma_n}\sup_{u \in A} \mathcal{\widetilde{I}}_{a,\lambda}(u)
 \end{equation}
 which satisfy evidently $-\infty<b_{a,\lambda,1}\leq b_{a,\lambda,2}\leq b_{a,\lambda,3} \leq \ldots$
The strict inequality holds being the functional bounded below.
 
Arguing as in  \cite{AzP} we get
the existence of sublevels  of the functional with arbitrary large genus.
\begin{lemma}\label{n17}
Given $n\in \mathbb{N},$ there is $\varepsilon=\varepsilon(n)>0$ satisfying $\gamma(\mathcal{\widetilde{I}}_{a,\lambda}^{-\varepsilon})\geq n.$
\end{lemma}
\begin{proof}
    Since $\mathcal{\widetilde{I}}_{a,\lambda}$ is continuous and even, it is for every $n\in \mathbb N,  \mathcal{\widetilde{I}}_{a,\lambda}^{-\varepsilon} \in\Sigma_{n}.$

    Fix any $n\in \mathbb{N}$ and let $F_n$ be the $n-$dimensional subspace of $H^{1}(\mathbb{R}^3)$
    spanned by $n$ smooth  functions supported in $\{x\in \mathbb R^{3}: K(x)>0 \}$ with disjoint supports.
Taking $z\in F_n$ with $\Vert z \Vert=1$ and $t\in (0,t_{\lambda,0}^{*})$, it is $\tau(\|t z\|) =1$
hence (recall \eqref{v:lemmaphi} of Lemma \ref{lemmaphi})
\begin{eqnarray*}\label{app6}
    \nonumber\mathcal{\widetilde{I}}_{a,\lambda}(t z)&=& \mathcal I_{a,\lambda}(t z) \nonumber\\
    &=&
    \frac{t^2}{2}\|z\|^{2} + \frac{t^4}{4}\int_{\mathbb{R}^3} Q(x)\phi_{z}^{a}z^2 -\lambda\frac{t^{p+1}}{p+1}\int_{\mathbb{R}^3} K(x) \vert z \vert^{p+1}
    -\frac{t^6}{6}\int_{\mathbb{R}^3} \vert z\vert^6\nonumber \\
    &\leq& c_{1}t^{2} + c_{2}t^{4} -\lambda\frac{t^{p+1}}{p+1}\int_{\{x : \, K(x)>0 \}} K(x) \vert z \vert^{p+1}    -\frac{t^6}{6}\int_{\{x : \, K(x)>0 \}} \vert z\vert^6.
        \end{eqnarray*}
        Now in $F_{n}$ all the norms are equivalent, hence
        \begin{eqnarray*}
        &\inf_{z\in F_{n}, \|z\| =1} |z|_{L^{6}(\{x : \, K(x)>0 \})}^{6} >0\\
        &\inf_{z\in F_{n}, \|z\| =1} \displaystyle\int_{\{x : \, K(x)>0 \}} K(x) |z|^{p+1}>0
        \end{eqnarray*}
and        consequently, for suitable constants $c_{3}, c_{4}>0$
       \begin{eqnarray*}
\widetilde{\mathcal I}_{a,\lambda}(tz)  &\leq& c_{1} t^2+ c_{2}t^4 -\lambda c_{3} t^{p+1}-c_{4}t^6. 
    \end{eqnarray*}

 Since  $p\in(0,1)$, there is $\varepsilon>0, \overline t\in(0,t_{\lambda,0}^{*})$ 
 such that 
 $\sup_{z\in F_{n}, \|z\|=1}\widetilde{\mathcal I}_{a,\lambda}(\overline t z)\leq -\varepsilon$; hence
$n=\gamma \{z\in F_{n}: \|z\|=\overline t \} 
\leq  \gamma  ( \widetilde{\mathcal I}_{a,\lambda}^{-\varepsilon})$
%
which concludes the proof.
\end{proof}

For $\lambda>0$ such that $\frac{1}{3}S^{\frac32}-\lambda^{\frac{2}{1-p}} |K|_{\frac{2}{1-p}}^{\frac{2}{1-p}}\mathfrak a_{p}>0$, namely
\begin{equation}\label{eq:L2}
0<\lambda< \frac{1}{|K|_{\frac{2}{1-p}}}\Big(\frac{1}{3 \mathfrak a_{p}}S^{\frac32}\Big)^{\frac{1-p}{2}}=:\Lambda_{2}
\end{equation} 
we guarantee, by Proposition \ref{cor:PS}, that the $(PS)$ condition is satisfied at negative
energy levels for $\mathcal I_{a,\lambda}$, then also for $\widetilde{\mathcal I}_{a,\lambda}$. So let now 
$$\Lambda:=\min\Big\{
 \Lambda_{1}, \Lambda_{2}
 \Big\},$$
(that of course does not depend on the parameter $a\geq0$), where $\Lambda_{1}, \Lambda_{2}$
are defined in \eqref{eq:L1} and \eqref{eq:L2}.
In this way, for $\lambda\in(0,\Lambda)$ we guarantee for $\mathcal I_{a,\lambda}$ the $(PS)$ condition at negative energy 
levels as well as the minimax scheme.
\begin{lemma}\label{no18}
    Assume $\lambda \in (0,\Lambda)$ and let  $n\geq1, r\geq0 $ such that
    $$b:=b_{a,\lambda,n}=b_{a,\lambda,n+1}=...=b_{a, \lambda, n+r}.$$
   Then $ \gamma(\emph K_b)\geq n+1.$ In particular $b_{a,\lambda,n}$ are critical levels.
\end{lemma}
\begin{proof}
Assume by contradiction  that $b=b_{a,\lambda, n}=\ldots=b_{a,\lambda,n+r}<0$ and $\gamma(\textrm K_b)\leq r$. Then there is a 
 symmetric and closed set $U$ with $\textrm K_b\subset U$ such that $\gamma(U)=\gamma(\textrm K_b)\leq r$
 and we can assume $U\subset \widetilde{\mathcal I}^{0}_{a,\lambda}$ since $b<0$.
 
Since the Palais-Smale condition holds, by the Deformation Lemma (see e.g.  \cite[Theorem A.4]{Rabinowitz}), there exists an odd homeomorphism 
$    \eta: H^1(\mathbb{R}^3) \rightarrow H^1(\mathbb{R}^3) $
such that $ \eta(\mathcal{\widetilde{I}}_{a,\lambda}^{b+\delta}\setminus U) \subset \mathcal{\widetilde{I}}_{a,\lambda}^{b-\delta}$
for some $\delta>0$, with $0<\delta <-b$. 
From
\begin{equation*}
    b=b_{a,\lambda,n+r}=\inf_{A\in \Sigma_{n+r}}\sup_{u\in A }\mathcal{\widetilde{I}}_{a,\lambda}(u),
\end{equation*}
 there exist $A\in \Sigma_{n+r}$ such that $\sup_{u\in A} \mathcal{\widetilde{I}}_{a,\lambda}(u)< b+\delta<0$, that is, 
 $A \subset \mathcal{\widetilde{I}}_{a,\lambda}^{b+\delta}\subset\widetilde{\mathcal I}_{a,\lambda}^{0}$. Hence  
 \begin{equation}\label{eq:contr}
  \eta(\overline{A\setminus U})\subset \overline{\eta(\mathcal{\widetilde{I}}_{a,\lambda}^{b+\delta}\setminus U) }\subset \mathcal{\widetilde{I}}_{a,\lambda}^{b-\delta} \quad \text{and}\quad \sup_{u\in \eta(\overline{A\setminus U})} \widetilde{\mathcal I}_{a,\lambda}(u) \leq b-\delta.
 \end{equation}
On the other hand, by the  properties of genus, we obtain 
$$ \gamma (\eta(\overline{A\setminus U}))\geq \gamma(\overline{A\setminus U})\geq \gamma(A)- \gamma(U) \geq n+r- r=n$$
so that $\eta (\overline{A\setminus U})\in \Sigma_n$. But then, 
using \eqref{eq:contr}
$$b=b_{a,\lambda,n} \leq \sup_{u\in \eta(\overline{A\setminus U})}\mathcal{\widetilde{I}}_{a,\lambda}(u)\leq b-\delta$$
which is a contradiction, and the proof is finished.
\end{proof}

\section{Proof of the results}\label{sec:Proofs}

\subsubsection{Proof of Theorem \ref{vegeta} and Theorem \ref{th:a=0}}

In view of the results in Section \ref{sec:PS} and Section \ref{sec:minmax} we deduce that
for every fixed $\lambda\in (0,\Lambda)$ and $a>0$
the numbers  $\{b_{a,\lambda,n}\}_{n\in \mathbb N}$ defined in \eqref{app8} 
are negative critical levels for $\widetilde{\mathcal I}_{a,\lambda}$, hence  for $\mathcal I_{a,\lambda}$,
which tends increasing to zero (see Theorem \cite[Theorem 0]{Kaji}). 
Moreover we can always assume by Theorem \cite[Theorem 1]{Kaji}
that the associated critical points $\{u_{a,\lambda,n}\}_{n\in \mathbb N}$ converge to zero
as $n$ goes to infinity.

Since $\mathcal I_{a,\lambda}(|u|) = \mathcal I_{a,\lambda}(u)$
the ground state solution $u_{a,\lambda,1}$ can be assumed positive.

Finally it is easy to see that the sequence of solutions found is compact.
In fact, 
\begin{itemize}
\item if  infinitely many of them are at the same critical level, let us call it $c\in [b_{a,\lambda,1}, 0)$, the
conclusion is obvious since $\textrm K_{c}$ is compact;
\item if  they are at infinitely many critical levels in 
$[b_{a,\lambda,1},0)$,
they  provide a Palais-Smale sequence (of solutions) and we know that such a sequence  is convergent 
by Proposition \ref{cor:PS}. 
\end{itemize}

Moreover by
\begin{eqnarray}
0&>& \mathcal I_{a,\lambda}(u_{a,\lambda,n}) -\frac{1}{6}\mathcal I'_{a,\lambda} (u_{a,\lambda,n})[u_{a,\lambda,n}] 
\nonumber\\
   &=& \frac{1}{3}\| u_{a,\lambda,n}\|^{2}   +\frac{1}{12}\int_{\mathbb{R}^3} Q(x)\phi_{u_{a,\lambda,n}}^{a} u_{a,\lambda,n}^{2}
    - \lambda \frac{5-p}{6(p+1)}  \int_{\mathbb{R}^3} K(x)\vert u_{a,\lambda,n}\vert^{p+1}\nonumber\\
    &\geq& \frac{1}{3}\Vert u_{a,\lambda,n}\Vert^{2}  -\lambda \frac{5-p}{6(p+1)} \vert K\vert_{\frac{2}{1-p}}\vert u_{a,\lambda,n}\vert^{p+1}_{2}\nonumber \\
&\geq& \frac{1}{3}\| u_{a,\lambda,n}\|^{2} - \lambda \frac{5-p}{6(p+1)} |K|_{\frac{2}{1-p}}\|u_{a,\lambda,n}\|^{p+1},
\label{eq:Ibd}
\end{eqnarray}
 the estimate on the norm of the solutions  follows, and consequently
  the squeezing.

This proves Theorem \ref{vegeta} and Theorem \ref{th:a=0}, since everything is true also when 
 $a=0$.

\subsubsection{Proof of Theorem \ref{th:} }

Indeed we have, for all $u_{a,\lambda}\in \mathcal S_{a,\lambda}$, by the same computation as above
\begin{eqnarray*}
\mathcal I_{a,\lambda}(u_{a,\lambda})&=& \mathcal I_{a,\lambda}(u_{a,\lambda}) -\frac{1}{6}\mathcal I'_{a,\lambda} (u_{a,\lambda})[u_{a,\lambda}] 
\nonumber\\
&\geq& \frac{1}{3}\| u_{a,\lambda}\|^{2} - \lambda \frac{5-p}{6(p+1)} |K|_{\frac{2}{1-p}}\|u_{a,\lambda}\|^{p+1}.
\end{eqnarray*}
Then by minimizing   the function $g(t) = c_{1}t^{2} - \lambda c_{2}t^{p+1}$
we see that
\begin{eqnarray*}
\inf_{S_{a,\lambda}} \mathcal I_{a,\lambda} &\geq & -  \lambda^{\frac{2}{1-p}}
\frac{1-p}{3} \left( \frac{5-p}{4}\right)^{\frac{2}{1-p}}(1+p)^{\frac{1+p}{ 1-p}} |K|_{\frac{2}{1-p}}\\
&=:& -\lambda^{\frac{2}{1-p}}  |K|_{\frac{2}{1-p}} \mathfrak c_{p} >-\infty.
\end{eqnarray*}
So that, the functional is bounded below on the solutions, and coercive if $\mathcal S_{a,\lambda}$ is unbounded.

\subsubsection{Proof of Theorem \ref{th:a>0}}
In fact, fixed $a\geq0, \lambda>0$ and taking  solutions $u_{a,\lambda}$,
by the same computations as in \eqref{eq:Ibd} with $\mathcal I_{a,\lambda}(u_{a,\lambda})\leq c$, we see that 
$$\|u_{a,\lambda}\|^{2} \leq \lambda\frac{5-p}{2(p+1)} |K|_{\frac{2}{1-p}}  \|u_{a,\lambda} \|^{p+1} + 3c$$
which proves the first statement.
 Moreover  any sequence of solutions  is a Palais-Smale sequence and,
if it is at an energy level 
$c<\frac{1}{3} S^{\frac32} - \lambda^{\frac{2}{1-p}} |K|^{\frac{2}{1-p}}_{\frac{2}{1-p}}\mathfrak a_{p}$, 
all the results of Section \ref{sec:PS} apply. In particular Proposition \ref{cor:PS} 
implies  that this sequence  is convergent in $H^{1}(\mathbb R^{3})$, and the limit is of course a solution with the same values of $a$ and $\lambda$.

The further uniform estimates on the solutions follows since such solutions form a Palais-Smale sequence
at the level $c$ and then by \eqref{eq:stimac}, with $\mu_{j_{0}} = \zeta_{j_{0}} = 0$, the estimate holds.

\subsubsection{Proof of Theorem \ref{th:squeezing}}

It follows by  the computations as in \eqref{eq:Ibd}, which are true  also for any solution.
So when $c\leq0$  we have the  bound
$$\|u_{a,\lambda}\|^{1-p} \leq \lambda\frac{5-p}{2(p+1)} |K|_{\frac{2}{1-p}}  $$
which is uniform in $a\geq0$,  and the squeezing holds for solutions of the Schr\"odinger-Bopp-Podolsky 
system as well for the Schr\"odinger-Poisson system.
The convergences in \eqref{eq:new0} and \eqref{eq:new} are obvious.

\subsubsection{Proof of Theorem \ref{th:convina} and Theorem \ref{th:GS1}}
Take general solutions $\{u_{a,\lambda}\}_{a\in(0,1]}$
and as before we get
\begin{eqnarray}
\frac{1}{3}S^{\frac32} -\lambda^{\frac{2}{1-p}} |K|_{\frac{2}{1-p}}^{\frac{2}{1-p}}\mathfrak a_{p}
&>&\mathcal{I}_{a,\lambda}(u_{a,\lambda})-\frac{1}{6}\mathcal{I}'_{a,\lambda}(u_{a,\lambda})[u_{a,\lambda}] \label{eq:questa}\\
&\geq&\frac{1}{3}\Vert u_{a,\lambda}\Vert^{2} -\lambda  \frac{5-p}{6(p+1)}|K|_{\frac{2}{1-p}}\Vert u_{a,\lambda}\Vert^{p+1}. \nonumber
\end{eqnarray}
Then  $\{u_{a,\lambda}\}_{a\in(0,1]}$ is bounded in $H^{1}(\mathbb R^{3})$, that 
  joint with \eqref{eq:questa},  gives that $\{\mathcal I_{\lambda, a}(u_{a,\lambda})\}_{a\in(0,1]}$ is bounded too. 
Then the family  $\{u_{a, \lambda}\}_{a\in(0,1]}$ ``behaves'' like a  bounded  Palais-Smale sequence,
although   now the functional is varying  with $a$.
However all the procedure made  in Section \ref{sec:PS} can be repeated just with very minor changes
(take $a_{n}\to 0$ and replace $u_{n}$ with $u_{a_{n}, \lambda}$)
and we get the strong convergence, namely, there is $\overline u_{\lambda}\in H^{1}(\mathbb R^{3})$ such that
\begin{equation}\label{eq:strongu}
\lim_{a\to 0} u_{a,\lambda} =  \overline u_{\lambda} \quad \text{ in }\quad H^{1}(\mathbb R^{3}).
\end{equation}

Let us  prove that $ \overline u_{\lambda}$ is a solution of the Schr\"odinger-Poisson equation \eqref{eq:eqSP}.
To this aim, we recall the following important result (see \cite[Lemma 6.1]{dAvGS}).
\begin{lemma}\label{lem:dS}
 Consider $f^0\in  L^{6/5}(\mathbb R^{3})$,  $\{ f^{a}\}_{ a\in (0,1)}\subset  L^{6/5}(\mathbb R^{3})$ and let 
 $$\phi^{0}\in D^{1,2}(\mathbb R^{3}) \hbox{ be the unique solution of }  -\Delta \phi = f^{0} \hbox{ in }\mathbb R^{3}$$
 and
 $$\phi^{a}\in \mathcal D\hbox{ be the unique solution of } -\Delta\phi + a^{2}\Delta^{2}\phi = f^{a}\hbox{ in }\mathbb R^{3}.$$ 
As  $a\to0$ we have:
\begin{enumerate}[label=(\roman*),ref=\roman*]
		\item\label{1ato0} if $f^{a}\rightharpoonup f^{0}$ in $L^{6/5}(\mathbb R^{3})$, then 
		$\phi^{a}\rightharpoonup \phi^{0}$ in $D^{1,2}(\mathbb R^{3})$; \smallskip
		\item\label{2ato0}  if $f^{a}\to f^{0}$ in $L^{6/5}(\mathbb R^{3})$, then $\phi^{a}\to \phi^{0}$ in $D^{1,2}(\mathbb R^{3})$ and $a\Delta \phi^{a} \to 0$ in $L^{2}(\mathbb R^{3})$.
	\end{enumerate}
\end{lemma}

From this we deduce that
\begin{equation}\label{eq:strongphi}
\lim_{a\to 0} \phi^{a}_{u_{a,\lambda}}= \phi^{0}_{ \overline u_{\lambda}} \ \  \text{ in }  \ D^{1,2}(\mathbb R^{3}) \quad\text{ and }
\quad 
\lim_{a\to 0}a\Delta \phi^{a}_{ u_{a,\lambda}} =0 \ \ \text{ in } \ L^{2}(\mathbb R^{3}),
\end{equation} 
where $-\Delta \phi^{0}_{ \overline u_{\lambda}} = 4\pi  \overline u_{\lambda}^{2}$.
%
%
If now $\varphi\in C_{c}^{\infty}(\mathbb R^{3})$, then
\begin{multline*}\label{eq:}
\int_{\text{supp\,}\varphi}\nabla u_{a,\lambda}\nabla \varphi + \int_{\text{supp\,}\varphi} 
u_{a,\lambda} \varphi +
\int_{\text{supp\,}\varphi}Q(x) \phi^{a}_{u_{a,\lambda}} u_{a,\lambda}\varphi =\\
\lambda \int_{\text{supp\,}\varphi} K(x) |u_{a,\lambda}|^{p-1}u_{a,\lambda} \varphi
+ \int_{\text{supp\,}\varphi} |u_{a,\lambda}|^{4}u_{a,\lambda} \varphi.
\end{multline*}
Passing to the limit as $a\to 0$ we achieve
\begin{multline*}
\int_{\text{supp\,}\varphi}\nabla  \overline u_{\lambda}\nabla \varphi + \int_{\text{supp\,}\varphi} 
 \overline u_{\lambda} \varphi +
\int_{\text{supp\,}\varphi}Q(x) \phi^{0}_{ \overline u_{\lambda}}  \overline u_{\lambda}\varphi =\\
\lambda \int_{\text{supp\,}\varphi} K(x) | \overline u_{\lambda}|^{p-1} \overline u_{\lambda} \varphi
+ \int_{\text{supp\,}\varphi} | \overline u_{\lambda}|^{4} \overline u_{\lambda} \varphi
\end{multline*}
implying that $ \overline u_{\lambda}$ solves \eqref{eq:eqSP}
 and then can be denoted  its own right with $u_{0,\lambda}$; so 
 $$\lim_{a\to 0} u_{a,\lambda} = u_{0,\lambda}\quad \text{ and } \quad \lim_{a\to 0} \phi^{a}_{u_{a,\lambda}} = \phi^{0}_{u_{0,\lambda}}$$
and from these, the  convergences in \eqref{eq:Ia0} follow.
Similarly  the convergences in \eqref{eq:I'a0} follow and Theorem \ref{th:convina} is proved.


\medskip

In particular, when we apply the Theorem \ref{th:convina}  to the family of ground states $\{u_{a,\lambda,1}\}_{a\in(0,1]}$
we get that the limit as $a$ goes to $0$ is a solution of \eqref{eq:eqSP}. Let us call it $\overline u_{0,\lambda}$.
But  by \eqref{eq:strongu} and the first limit in \eqref{eq:strongphi}
we have
$$\lim_{a\to 0}\mathcal I_{a,\lambda}(u_{a,\lambda,1}) = \mathcal I_{0,\lambda}(\overline u_{0,\lambda}).$$
On the other hand, since $\mathcal I_{a,\lambda}(u) \leq \mathcal I_{0,\lambda}(u)$
for every $a>0$, we obtain
\begin{eqnarray*}
\mathcal  I_{a,\lambda}(u_{a,\lambda,1})
\leq\mathcal I_{a,\lambda}(u_{0,\lambda,1})
\leq \mathcal I_{0,\lambda}(u_{0,\lambda,1}) 
=b_{0,\lambda,1}
\end{eqnarray*}
and hence,   we infer
\begin{eqnarray*}
\mathcal I_{0,\lambda}(\overline u_{0,\lambda}) = \limsup_{a\to 0}\mathcal{I}_{a,\lambda}(u_{a,\lambda,1}) \leq b_{0,\lambda,1}
\end{eqnarray*}
implying that $\overline u_{0,\lambda}$ is a ground state for \eqref{eq:eqSP}, and then it can be denoted with 
$u_{0,\lambda,1}$. This finishes the proof of Theorem \ref{th:GS1}.

\bigskip

{\bf Acknowledgements.}

H. M. Santos Damian was supported by  CNPq grant n. 140423/2020-6.
G. Siciliano was supported by Capes, CNPq, FAPDF Edital 04/2021 - Demanda Espont\^anea, 
Fapesp grants no. 2022/16407-1 and 2022/16097-2 (Brazil),  
PRIN PNRR, P2022YFAJH “Linear and Nonlinear PDEs: New directions and applications”, and 
INdAM-GNAMPA project n. E5324001950001 ``Critical and limiting phenomena in nonlinear elliptic systems'' (Italy).

\medskip

{\bf Data availability. } Data sharing not applicable to this article as no datasets were generated or analysed during
the current study.

\medskip

{\bf Conflict of interest. }The authors declare that they have no conflict of interest.

\medskip

{\bf Author contributions. } The authors contributed equally to the writing of this article. All authors
read and approved the final manuscript.

\end{document}